\documentclass[leqno]{siamart1116}
\pdfoutput=1
\usepackage{amssymb}
\usepackage{graphicx} 
\usepackage{xspace}
\usepackage{multirow}
\usepackage{bbm}
\usepackage{mydef}
\usepackage{algorithm}
\usepackage{algpseudocode}
\usepackage[numbers]{natbib}
\let\cite=\citet

\begin{document}

\newcommand\footnotemarkfromtitle[1]{%
\renewcommand{\thefootnote}{\fnsymbol{footnote}}%
\footnotemark[#1]%
\renewcommand{\thefootnote}{\arabic{footnote}}}

\title{Second-order invariant domain preserving approximation of the
  Euler equations\\ using convex limiting\footnotemark[1]}

\author{Jean-Luc Guermond\footnotemark[2]
\and Murtazo Nazarov\footnotemark[3]
\and Bojan Popov\footnotemark[2]
\and Ignacio Tomas\footnotemark[2]}
\date{Draft version \today}

\maketitle

\renewcommand{\thefootnote}{\fnsymbol{footnote}} \footnotetext[1]{
 This material is based upon work supported in part by the National
  Science Foundation grants DMS-1619892 and DMS-1620058, by the Air
  Force Office of Scientific Research, USAF, under grant/contract
  number FA9550-15-1-0257, and by the Army Research Office under
  grant/contract number W911NF-15-1-0517.

  Sept. 30th 2017. Submitted for publication to SIAM SISC.}
\footnotetext[2]{Department of Mathematics, Texas A\&M University 3368
  TAMU, College Station, TX 77843, USA.}  \footnotetext[3]{Division of
  Scientific Computing, Department of
  Information Technology, Uppsala University, Uppsala, Sweden}
\renewcommand{\thefootnote}{\arabic{footnote}}

\begin{abstract}
  A new second-order method for approximating the compressible Euler
  equations is introduced.  The method preserves all the known
  invariant domains of the Euler system: positivity of the density,
  positivity of the internal energy and the local minimum principle on
  the specific entropy. The technique combines a first-order,
  invariant domain preserving, Guaranteed Maximum Speed method using a
  Graph Viscosity (GMS-GV1) with an invariant domain violating,
  but entropy consistent, high-order method.  Invariant domain
  preserving auxiliary states, naturally produced by the GMS-GV1
  method, are used to define local bounds for the high-order method
  which is then made invariant domain preserving via a {\em convex
    limiting} process. Numerical tests confirm the second-order
  accuracy of the new GMS-GV2 method in the maximum norm, where 2
  stands for second-order.  The proposed convex limiting is generic
  and can be applied to other approximation techniques and other
  hyperbolic systems.
\end{abstract}



\pagestyle{myheadings} \thispagestyle{plain} \markboth{J.-L. GUERMOND,
  M. NAZAROV, B. POPOV, I. TOMAS}{Convex limiting}

\section{Introduction}\label{sec:intro} 
The objective of the present work is to present an approximation
technique for the compressible Euler equations that is explicit in
time, second-order accurate in space and time, and invariant domain
preserving. The method is presented in the context of continuous
finite elements, but it is quite general and can be applied to other
discretization settings like discontinuous Galerkin and finite
volume techniques.  Like in many other high-order approximation methods, the
proposed technique consists of combining a first-order, invariant
domain preserving, and entropy satisfying approximation with a
high-order entropy consistent approximation. The high-order method is
made invariant domain preserving and formally entropy compliant by adapting the
artificial viscosity through a limitation process. One key novelty is
that the density, the internal energy and the specific entropy are
limited by using bounds that are in the domain of dependence of the
local data.  Another novelty is that the limiting bounds are local for
all the quantities, and these bounds are naturally satisfied by the
low-order solution.

The so-called Flux Transport Corrected method (FCT) introduced in
\cite{Boris_books_JCP_1973} for approximating the one-dimensional mass
conservation equation and later generalized to multi dimensions in
\cite{Zalesak_1979} are among the first successful techniques to
produce second-accuracy while imposing pointwise bounds such as
positivity of the density.  This methodology can be used to preserve
the maximum principle for any scalar conservation equation. We refer
the reader to \cite{TurekKuzmin2002,KuzminLoehnerTurek2004} for
reviews on FCT. Unfortunately, the FCT algorithm, as proposed in the
above references, is not well suited to enforce constraints that are
not affine. For instance it cannot be (easily) applied to guarantee
the positivity of the specific internal energy and the minimum
principle on the specific entropy since these are quasiconcave
constraints. (The said constraints can be made concave by
multiplication by the density as it is shown in
\S\ref{Sec:convexity_based_limitation}.)  In the context of finite
volumes, efficient second-order limitation techniques for the specific
internal energy and the specific entropy have been first proposed in
\cite{Khobalatte_Perthame_1994}, \cite{Perthame_Youchun_1994}, and
\cite{Perthame_Shu_1996}. These ideas have been extended to
Discontinuous Galerkin framework in a series of papers by
\cite{Zhang_Shu_2010,Zhang_Shu_2012}, \cite{Jiang_Liu_2017}.  The key
argument that is common to
\citep{Khobalatte_Perthame_1994,Perthame_Shu_1996,Perthame_Youchun_1994}
is to rely on convex combinations and concavity. In the present paper
we are going to build on these ideas and propose a general FCT-like
post-processing methodology to enforce general concave constraints for
the Euler equations.  Instead of limiting slopes or reconstructed
approximations, we adopt an algebraic point of view similar to FCT.
The method is presented in the context of continuous finite elements,
but since it is algebraic, it can be applied to finite volumes and
discontinuous Galerkin approximation techniques as well.

The paper is organized as follows. The problem is formulated in
\S\ref{Sec:preliminaries}. The finite element setting and the notation
are also introduced in this section. The time and space approximation
using continuous finite elements is described in
\S\ref{Sec:time_space_approximation}. Both the first-order, invariant
domain preserving, entropy satisfying method and the high-order
invariant domain violating method are detailed in this section. The
bulk of the novel material is reported in
\S\ref{Sec:convexity_based_limitation}. The main results of this
section are the local bounds given in
\eqref{rhominmaxdef}--\eqref{minimum_internal_energy} together with
Lemma~\ref{Lem:quasiconcave} and Lemma~\ref{Lem:compute_lij}.  The
performance of the proposed method is illustrated in
\S\ref{Sec:Numerical_illustrations}. 

\section{Preliminaries} \label{Sec:preliminaries} We introduce in this
section the Euler equations and the finite element setting.  Some
important properties of the Euler equations that are used later in the
paper are also recalled. The reader who is familiar with the Euler
equations, invariant domains, and the finite element theory is invited
to jump to \S\ref{Sec:time_space_approximation}.

\subsection{The Euler equations}
Let $d$ be the space dimension and $\Dom$ be an open polyhedral
domain in $\Real^d$. We consider the compressible Euler equations in
conservative form in $\Real^d$:
\begin{subequations}\label{Euler}
\begin{align}
  &\partial_t \rho + \DIV \bbm = 0, \label{Euler:mass}\\
  &\partial_t \bbm + \DIV (\bv\otimes \bbm) + \GRAD p  = 0, \label{Euler:momentum}\\
  &\partial_t E + \DIV (\bv(E+p)) = 0,\label{Euler:energy}\\
  & \rho(\bx,0)=\rho_0,\qquad \bbm(\bx,0)=\bbm_0,\qquad E(\bx,0)=E_0.
\label{Euler:init}
\end{align}
\end{subequations}
The independent variables are  $(\bx,t)\in \Dom\CROSS\Real_+$.
The dependent variables, henceforth called conservative variables, are
the density, $\rho$, the momentum, $\bbm$ and the total energy,
$E$. The quantity $\bv:=\rho^{-1} \bbm$ is the velocity of the fluid
particles. The pressure, $p$, is given by the equation of state which
we assume to be derived from a specific entropy,
$s:\Real_+ \CROSS \Real_+ \to \Real$, defined through the
thermodynamics identity:
$T \diff s := \diff e - p\rho^{-2} \diff \rho$, where
$e := \rho^{-1}E - \frac12 \bv^2$ is the specific internal energy.
For instance it is common to take
$s(\rho,e) - s_0=\log(e^{\frac{1}{\gamma-1}}\rho^{-1})$ for a polytropic ideal gas.
Using the notation $s_e(\rho,e):=\frac{\partial s}{\partial e}(\rho,e)$ and
$s_\rho(\rho,e) :=\frac{\partial s}{\partial \rho}(\rho,e)$, the equation of state
then takes the form $p:= - \rho^2 s_\rho s_e^{-1}$.  To
simplify the notation we introduce the conservative variable
$\bu:=(\rho,\bbm,E)\tr \in \Real^{d+2}$ and the flux
$\bef(\bu) := (\bbm, \bv\otimes \bbm + p\polI, \bv(E+p))\tr \in
\Real^{m \times (d+2)}$,
where $\polI$ is the $d\CROSS d$ identity matrix. To avoid 
an abuse of notation and ambiguities, we introduce
\begin{equation}
  \Phi(\bu) := s(\rho, e(\bu)),
\end{equation}
where $e(\bu) := \rho^{-1}E - \frac{\bbm^2}{2\rho}$,
$\bbm^2:= \|\bbm\|_{\ell^2}^2$ and $\|\cdot\|_{\ell^2}$ is the
Euclidean norm, \ie $\Phi$ is the specific entropy expressed as a function of
the conservative variables. Finally, we call internal
energy the quantity $\varepsilon := \rho e= E -\frac12 \rho\bv^2$. 

The convention adopted in this paper is that for any vectors $\ba$,
$\bb$, with entries $\{a_k\}_{k=1,\dots,d}$, $\{b_k\}_{k=1,\ldots,d}$,
the following holds: $(\ba\otimes\bb)_{kl} = a_kb_l$ and
$\DIV\ba = \sum_{k=1,\dots,d}\partial_{x_k} a_{k}$,
$(\GRAD\ba)_{kl}=\partial_{x_l} a_k$.
$\ba\SCAL\GRAD =\sum_{k=1,\dots,d} a_k\partial_{x_k}$. Moreover, for
any second-order tensor $\polg$, with entries
$\{\polg_{kl}\}_{k,l=1,\dots,d}$, we define
$(\DIV\polg)_k = \sum_{l=1,\dots,d} \partial_{x_l} \polg_{kl}$,
$(\polg\ba)_k= \sum_{l=1,\dots,d}\polg_{kl} a_l$,
$(\ba\tr \polg)_l= \sum_{l=1,\dots,d}a_k \polg_{kl}$.

To avoid technicalities regarding boundary conditions, we assume that
either periodic boundary conditions are enforced, or the initial data
is compactly supported and in that case we are interested in the
solution before the domain of influence of $(\rho_0,\bbm_0,E_0)$
reaches the boundary of $\Dom$, \ie homogeneous Dirichlet
boundary conditions are enforced. (Some details are given in \S\ref{Sec:Mach3}
on how to deal with realistic boundary conditions.)

\subsection{Intrinsic properties}\label{SubSNatConst} 
The well-posedness of \eqref{Euler} is an extremely difficult question
that is far beyond the scope of the present paper. But to make sense
of the approximation techniques to be presented in the rest of the
paper we are going to rely on the notion of solution of
one-dimensional Riemann problems which is more tractable.  For any
unit vector $\bn\in \Real^d$, we consider the following Riemann
problem:
\begin{equation}
 \label{def:Riemann_problem} 
  \partial_t \bu + \partial_x (\bef(\bu)\SCAL\bn)=0, 
\quad  (x,t)\in \Real\CROSS\Real_+,\qquad 
\bu(x,0) = \begin{cases} \bu_L, & \text{if $x<0$} \\ \bu_R,  & \text{if $x>0$}, \end{cases}
\end{equation}
and assume that there exists a so-called admissible set $\calA$ such
that, for any pair of states $(\bu_L,\bu_R)\in \calA\CROSS \calA$,
this problem has a unique physical (entropy) solution henceforth
denoted by $\bu(\bn,\bu_L,\bu_R)(x,t)$. This assumption holds true, at
least for the covolume equation of state, with
$\calA:=\{\bu\st \rho> 0, e> 0\}$, and we refer the reader to
\cite[Thm.~II.3.1]{Godlewski_Raviart_1996} for a similar statement
with more general equations of states. We now introduce notions of
invariant sets and invariant domains. (Our definition is slightly
different from those in
\cite{Chueh_Conley_Smoller_1977,Hoff_1985,Frid_2001}.)
\begin{definition}[Invariant set] \label{Def:invariant_set} We say that a set
  $B\subset \calA\subset \Real^m$ is invariant for
  \eqref{Euler} if for any pair $(\bu_L,\bu_R)\in B\CROSS B$,
 any unit vector $\bn\in\Real^d$, and any $t>0$, the space average over the Riemann fan of the entropy
  solution $\bu(\bn,\bu_L,\bu_R)(x,t)$
 remains in $B$ for all $t >0$.
\end{definition}
We are also going to make use of the notion of invariant domain for an
approximation process.  Let $\bX_h \subset L^1(\Real^d;\Real^m)$ be a
finite-dimensional approximation space and let
$S_h : \bX_h \ni \bu_h\longmapsto S_h(\bu_h)\in \bX_h$ be a mapping
over $\bX_h$. (Think of $S_h$ as being a one-time-step approximation
of \eqref{Euler}.)  Henceforth we abuse the language by saying that a
member of $\bX_h$, say $\bu_h$, is in the set $B\subset \Real^m$ when
actually we mean that $\{\bu_h(\bx) \st \bx \in\Real^d \}\subset B$.
\begin{definition}[Invariant domain]\label{Def:invariant_domain}
  A convex invariant set $B \subset \calA\subset \Real^m$ is said to
  be an invariant domain for the mapping $S_h$ if and only if for any
  state $\bu_h$ in $B$, the state $S_h(\bu_h)$ is also in $B$.
\end{definition}
It is known that the set
\begin{equation}
A_{s^{\min}}:= \{ (\rho,\bbm,E) \st \rho> 0,e > 0,s\ge s^{\min}\}
\end{equation}
is invariant for the Euler system for any $s^{\min}\in \Real$.  It is
also established in \cite[Thm.~8.2.2]{Serre_2000_II} that the set
$A_{s^{\min}}$ is convex, and it is shown in \cite[Thm.~7 and
8]{Frid_2001} that it is an invariant domain for the Lax-Friedrichs
scheme. The finite element method introduced in \cite{Guer2016} also
satisfies this invariant domain property; this finite element
construction is recalled in \S\ref{Sec:low_order}.  It is generally
admitted in the literature that physical solutions to \eqref{Euler}
should satisfy entropy inequalities.  More specifically, let
$f:\Real\to \Real$ be twice differentiable and be such that
\begin{equation}
  f'(s) > 0,\qquad f'(s) c_p^{-1} - f''(s) > 0,
\qquad \forall (\rho,e) \in \Real_+^2, \label{def_of_f}
\end{equation}
where $c_p(\rho,e)=T\partial_T s(p,T)$ is the specific heat at
constant pressure. It is shown in
\cite[Thm.~2.1]{Harten_Lax_Levermore_Morojoff_1998} that $\rho f(s)$
is strictly concave with respect to the conservative variables if and
only if \eqref{def_of_f} holds, \ie $\calA\ni\bu\mapsto \rho f(\Phi(\bu))$ is convex. 
Note that the so-called physical entropy
$S(\bu):=\rho s$ is obtained by setting $f(s)=s$. We say that a weak
solution to \eqref{Euler} is an entropy solution if it satisfies the
following inequality in the weak sense for every generalized entropy:
\begin{equation}
\partial_t(\rho f(s)) +\DIV (\bbm f(s)) \ge 0. \label{entropy_inequality}
\end{equation}
In particular it is known (at least for $\gamma$-law equation of
state) that the entropy solution to the Riemann problem
\eqref{def:Riemann_problem} satisfy \eqref{entropy_inequality}.
 Both the Lax-Friedrichs scheme and the finite element method introduced in 
 \citep{Guer2016} satisfy a discrete version of \eqref{entropy_inequality} for 
every generalized entropy. 

The objective of the present work is to construct an explicit,
second-order continuous finite element method that is consistent with
\eqref{entropy_inequality} and for which $A_{s^{\min}}$ is an
invariant domain.

\subsection{Finite element setting}
We are going to approximate the solution of \eqref{Euler} with
continuous Lagrange finite elements.  We introduce for this purpose a
shape-regular sequence of matching meshes $\famTh$ and assume that the
elements in each mesh are generated from a small collection of
reference elements denoted $\wK_1,\dots,\wK_\varpi$. In two space
dimensions for instance, the mesh $\calT_h$ could be composed of a
combination of parallelograms and triangles (\ie $\varpi=2$). In three
space dimensions, $\calT_h$ could also be composed of a combination of
triangular prisms, parallelepipeds, and tetrahedra ($\varpi=3$). Given
$K\in \calT_h$, the geometric transformation mapping $\wK_r$ to
$K\in \calT_h$ is denoted $T_K : \wK_r \longrightarrow K$.  We are
going to construct the approximation space by using some reference
Lagrange finite elements
$\{(\wK_r,\wP_r,\wSigma_r)\}_{1\le r\le \varpi}$, where the objects
$(\wK_r,\wP_r,\wSigma_r)$ are Ciarlet triples (we omit the index
$r\in\intset {1}{\varpi}$ in the rest of the paper to simplify the
notation).  Given a reference Lagrange element $(\wK,\wP,\wSigma)$, we
denote by $\{\wbx_l\}_{l\in \calL}$ the reference Lagrange nodes and
by $\{\wtheta_l\}_{l\in \calL}$ the reference shape functions, \ie
$\text{card}(\calL) = \text{dim}(\wP)=:\nf$ (note that the index $r$
has been omitted). $\wP$ is the reference approximation space (usually
a scalar-valued polynomial space) and $\wSigma$ is the set of the
Lagrange degrees of freedom.  Let $\polP_{l,d}$ be the vector space
composed of the $d$-variate polynomials of degree at most $l$. We
henceforth assume that there is $k\ge 1$ such that
$\polP_{k,d}\subset \wP.$ The reference degrees of freedom
$\{\wsigma_l\}_{l\in \calL}$ are such that
$\wsigma_l(\wp)=\wp(\wbx_l)$ for all $l\in \calL$ and all
$\wp\in\wP$. By definition $\wtheta_l(\wbx_{l'}) =\delta_{l l'}$, for
all $l.l'\in\calL$, which is turn implies the partition of unity
property: $\sum_{l\in\calL}\wtheta_l(\wbx) =1$, for all $\wbx \in\wK$.
Setting $\{\ww_l:=\int_{\wK}\theta_l\diff x\}_{l\in \calL}$, the
following quadrature is $(k+1)$-th order accurate since it is exact
for all $\wp \in \wP$:
$\int_\wK \wp(\wbx)\diff \wbx = \sum_{l\in \calL} \wp(\wbx_l) \ww_l$.
We henceforth assume that
\begin{equation}
\int_\wK \wtheta_l\diff \wbx =: \ww_l>0,\qquad l\in \calL. \label{positivity_lumped_mass}
\end{equation} 
There are numerous reference finite elements satisfying
\eqref{positivity_lumped_mass}.  For instance all the elements based
on $\polQ_{k,d}$ polynomials using tensor
product of Gauss-Lobatto points on quadrangles or hexahedra satisfy
\eqref{positivity_lumped_mass}. The question is slightly more delicate
for $\polP_{k,d}$ polynomials on simplices, but one can use Fekete
points (see \cite{Taylor_Wingate_2000}) or various variations thereof
for $k\ge 3$ (see \cite{Hesthaven_1998,Warburton_2006}). Note that the
standard $\polP_{1,d}$ Lagrange element satisfies
\eqref{positivity_lumped_mass} but $\polP_{2,d}$ does not for
$d\ge 2$.

We now introduce approximation spaces constructed as usual by using
the pullback by the geometric transformation. More precisely we define
the following scalar-valued and vector-valued finite element spaces:
\begin{align} \label{eq:Xh}
P(\calT_h) &=\{ v\in \calC^0(\Dom;\Real)\st 
v_{|K}{\circ}T_K \in \wP,\ \forall K\in \calT_h\},\qquad 
\bP(\calT_h) = [P(\calT_h)]^{d+2}.
\end{align}  
The global shape functions in $P(\calT_h)$, which we recall form a
basis of $P(\calT_h)$, are denoted by $\{\varphi_i\}_{i\in\calI}$, \ie
$\text{card}(\calI)=\text{dim}(P(\calT_h))$. The global Lagrange nodes
are denoted $\{\bx_i\}_{i\in \calI}$. Upon introducing the
connectivity array $\jj : \calT_h\CROSS \calL \longrightarrow \calI$,
we have $\varphi_{\jj(l,K)}(\bx) = \wtheta_l((T_K)^{-1}(\bx))$, and
$\bx_{\jj(K,l)} = T_K(\wbx_l)$, for all $l\in\calL$ and all
$K\in \calT_h$.  This implies that $\varphi_i(\bx_j)=\delta_{ij}$.
The local partition of unity property implies that
\begin{align}
\sum_{i\in\calI}\varphi_i(\bx) =1,\
\quad \forall \bx \in\Dom.  \label{partition_unity}
\end{align}

Upon defining $m_i:= \int_{\Dom} \varphi_i(\bx)\diff \bx$, the above
definitions imply that the following quadrature is $(k+1)$-th order
accurate $\int_\Dom v(\bx)\diff \bx = \sum_{i\in\calI} m_i v(\bx_i)$
since it is exact for all $v\in P(\calT_h)$.  The matrix with entries
$\int_{\Dom} \varphi_i(\bx)\varphi_j(\bx)\diff \bx$ is called the
consistent mass matrix and denoted by
$\calM\in \Real^{\Nglob\CROSS \Nglob}$.  Using the above quadrature
and the property $\varphi_i(\bx_j)=\delta_{ij}$, the integral
$\int_{\Dom} \varphi_i(\bx)\varphi_j(\bx)\diff \bx$ can be
approximated by $m_i$.  Note that \eqref{partition_unity} implies that
$\sum_{j\in\calI} m_{ij} = m_i$.  We henceforth denote by $\calM^L$
the diagonal matrix with entries $(m_i)_{i\in\calI}$ and refer to
$\calM^L$ as the lumped mass matrix. Note that \eqref{positivity_lumped_mass} 
implies that $m_i>0$ for all $i\in \calI$.

We denote by $\Dom_i$ the support of $\varphi_i$. Let $E$ be a union
of cells in $\calT_h$; we define $\calI(E)$ to be the set containing
the indices of all the shape functions whose support on $E$ is of
nonzero measure. The sets $\calI(K)$ and $\calI(\Dom_i)$ will
be invoked regularly; note in particular that the partition of unity
can be rewritten $\sum_{i\in\calI(K)} \varphi_i(\bx) =1$ for all
$\bx\in K$.

Upon denoting by $\|\cdot\|_{\ell^2}$ the Euclidean norm in $\Real^d$,
we introduce the following two quantities which will play an
important in the rest of paper:
\begin{equation} 
  \bc_{ij} := \int_\Dom \varphi_i \GRAD\varphi_j \dif
  x,\qquad \bn_{ij}:= \frac{\bc_{ij}}{\|\bc_{ij}\|_{\ell^2}}\qquad
  i,j\in\calI.
\label{def_of_cij}
\end{equation} 
Note that \eqref{partition_unity} implies
$\sum_{j\in\intset{1}{\Nglob}} \bc_{ij} =\mathbf{0}$. Furthermore, if
either $\varphi_i$ or $\varphi_j$ is zero on $\partial\Dom$, then
$\bc_{ij} = - \bc_{ji}$. In particular we have
$\sum_{i\in\intset{1}{\Nglob}} \bc_{ij} =\mathbf{0}$ if $\varphi_j$ is
zero on $\partial\Dom$.  This property will be used to establish
conservation.

\section{Time and space approximation}
\label{Sec:time_space_approximation}
We describe in this section the first-order technique and the
higher-order technique that will be used to construct the second-order
invariant domain preserving and entropy compliant method which is the
object of the present paper. The discussion is restricted to the
forward Euler time stepping since higher-order accuracy in time is
trivially achieved by using Strong Stability Preserving Runge-Kutta
time stepping.  All the numerical tests reported in
\S\ref{Sec:Numerical_illustrations} are done with the SSP RK(3,3)
method (three stages, third-order), see
\cite[Eq.~(2.18)]{Shu_Osher1988} and
\cite[Thm.~9.4]{Kraaijevanger_1991}.

\subsection{Low-order approximation (GMS-GV1)} \label{Sec:low_order}
The low-order method that we are going to use is fully described in
\citep{Guer2016} and is henceforth referred to as GMS-GV1 for
Guaranteed Maximum Speed method with first-order Graph Viscosity. Let
$\bu_{h}^0 = \sum_{i=1}^\Nglob \bsfU_i^0 \varphi_i \in \bP(\calT_h)$
be a reasonable approximation of the initial data $\bu_0$. Let $\dt$
be the time step, $t_n$ be the current time, and let us set
$t_{n+1}=t_n+\dt$ for some $n\in \polN$. Letting
$\bu_h^n=\sum_{i\in\calI}^\Nglob \bsfU_i^n \varphi_i$ be the space
approximation of $\bu$ at time $t_n$, we estimate the low-order
approximation
$\bu_h^{L,n+1}=\sum_{i\in\calI} \bsfU_i^{L,n+1} \varphi_i$ by setting
\begin{equation}
\label{def_dij_scheme}
  \frac{m_i}{\dt}(\bsfU_i^{L,n+1}-\bsfU_i^n)
  + \sum_{j\in \calI(\Dom_i)} \bef(\bsfU_j^n)\SCAL \bc_{ij}
  - d_{ij}^{L,n} (\bsfU^n_j-\bsfU^n_i)  = 0.
\end{equation}
The graph viscosity coefficients $d_{ij}^{L,n}$ are defined for all
$i \not= j \in \calI$ by
\begin{equation}
\label{Def_of_dij_I}
  d_{ij}^{L,n} := \max(\lambda_{\max}(\bn_{ij},\bsfU_i^n,\bsfU_j^n)
  \|\bc_{ij}\|_{\ell^2},
\lambda_{\max}(\bn_{ji},\bsfU_j^n,\bsfU_i^n) \|\bc_{ji}\|_{\ell^2}),
\end{equation}
where 
$\lambda_{\max}(\bn,\bsfU_L,\bsfU_R)$  is the maximum wave speed in the Riemann problem:
\begin{equation}
\partial_t\bw + \partial_x\bef^{\text{1D}}(\bw) =0,
\label{one_d_riemann_with_n}
\end{equation}
with data $\bw_L:=(\rho_L,m_L,\mathcal{E}_L)\tr$,
$\bw_R:=(\rho_R,m_R,\mathcal{E}_R)\tr$, where $m:= \bbm\SCAL \bn$,
$v:=m/\rho$, $\bbm^\perp := \bbm - (\bbm\SCAL \bn)\bn$,
$\mathcal{E} := E - \frac12 \frac{\|\bbm^\perp\|_{\ell^2}^2}{\rho}$,
and flux $\bef^{\text{1D}}(\bw) := (m,vm+p,v(\calE+p))\tr$.  A
guaranteed upper bound on $\lambda_{\max}(\bn,\bsfU_L,\bsfU_R)$ is
given in \citep[Rem.~2.8]{Guer2016} and \cite[Appendix
C]{Guermond_Popov_Fast_Riemann_2016} for the covolume equation of
state $p(1-b\rho)=(\gamma-1)e\rho$, with $b\ge 0$; the case $b=0$
corresponds to an ideal gas.  Let $v_L := \bbm_L\SCAL\bn/\rho_L$,
$v_R:= \bbm_R\SCAL\bn/\rho_R$ be the left and right speed and let
$c_L$, $c_R$ the left and right speed of sound. We recall in passing
that the widely used estimate $\max(|v_L|+c_L,|v_R|+c_R)$ is {\em not}
a a guaranteed upper bound of $\lambda_{\max}(\bn,\bsfU_L,\bsfU_R)$.
For the reader's convenience, we now recall the upper bound proposed
in \cite{Guermond_Popov_Fast_Riemann_2016}.  The definition of the
maximum speed
$\lambda_{\max}(\bn,\bsfU_L,\bsfU_R)=\lambda_{\max}(\bw_L,\bw_R)$ is
as follows:
\begin{equation}
\lambda_{\max}(\bw_L,\bw_R)=
\max((\lambda_1^-)_-,(\lambda_3^+)_+),
\end{equation}
where $\lambda_1^-$ and $\lambda_3^+$ are the two extreme wave speeds
enclosing the Riemann fan of the one-dimensional
problem~\eqref{one_d_riemann_with_n}; these two extreme wave speeds are
given by
\begin{align}
\lambda_1^-(p^*)=v_L - c_L\left(
1+\frac{\gamma+1}{2\gamma}\left(\frac{p^*-p_L}{p_L}\right)_+
\right)^\frac12, 
\\
\lambda_3^+(p^*)=v_R + c_R\left(
1+\frac{\gamma+1}{2\gamma}\left(\frac{p^*-p_R}{p_R}\right)_+
\right)^\frac12,
\end{align}
where the intermediate pressure $p^*$ is obtained by solving a
nonlinear problem.  Here we use the notation $z_+:=\max(0,z)$ and we
recall that the local sound speed for the covolume gas is
$c=\sqrt{\frac{\gamma p}{\rho(1-b\rho)}}$.  Note that the exact value
of value of $\lambda_{\max}(\bw_L,\bw_R)(p^*)$ is a monotone
increasing function of $p^*$.  Therefore, instead of computing the
exact pressure $p^*$ which requires an iterative process, one can use
an explicit upper bound $\tilde{p}^*\ge p^*$. Then, a guaranteed upper
bound on $\lambda_{\max}(\bw_L,\bw_R)(p^*)$ is
$\lambda_{\max}(\bw_L,\bw_R)(\tilde{p}^*)$.  One such upper bound
valid for $1<\gamma\le \frac53$ is established in
\cite[Lem.~4.3]{Guermond_Popov_Fast_Riemann_2016} and the value of
$\tilde{p}^*$ in question is given by the so-called two-rarefaction
approximation:
\begin{equation}
\tilde{p}^*= \left(
\frac{c_L(1-b\rho_L)+c_R(1-b\rho_R)-\frac{\gamma-1}{2}(v_R-v_L)}
{ c_L(1-b\rho_L)\ p_L^{-\frac{\gamma-1}{2\gamma}}
+ c_R(1-b\rho_R)\ p_R^{-\frac{\gamma-1}{2\gamma} } }
\right)^{\frac{2\gamma}{\gamma-1}}.
\end{equation}

\begin{remark}[Graph Viscosity]
  Note that the expression $- d_{ij}^{L,n} (\bsfU^n_j-\bsfU^n_i)$ in
  \eqref{def_dij_scheme} is called weighted Graph Laplacian in graph
  theory. We call the weights  $d_{ij}^{L,n}$ Graph Viscosity (or artificial viscosity).
\end{remark}

\begin{remark}[Other discretizations]
  Note that the expression \eqref{def_dij_scheme} that is used to
  compute the update $\bsfU_i^{n+1}$ is quite generic; many other
  discretizations of the Euler equations can be put in this abstract
  form. The notion of continuous finite element only intervenes in the
  definition of the vectors $\bc_{ij}$, the index set $\calI(\Dom_i)$,
  and the lumped mass matrix coefficients $m_i$. Other discretizations
  lead to other forms for $\bc_{ij}$, $\calI(\Dom_i)$, and 
  $m_i$. Almost everything that is said in the rest of the paper can be applied 
 to these discretizations as well.
\end{remark}

\subsection{The intermediate limiting
  states} \label{Sec:intermediate_limiting_states} We now deduce from
\eqref{def_dij_scheme} intermediate local states that will be useful
to limit the yet to be defined high-order solution.  Using that
$\sum_{j\in \calI(\Dom_i)} \bc_{ij}=0$, we rewrite
\eqref{def_dij_scheme} as follows:
\begin{align*}
  \frac{m_i}{\dt}\bsfU_i^{L,n+1}\! 
= \bsfU_i^n\Big(\frac{m_i}{\dt} - \!\!\!\!\sum_{i\ne j\in\calI(\Dom_i)}\!\!\!\!2d_{ij}^{L,n}\Big) 
  + \!\!\!\sum_{i\ne j\in\calI(\Dom_i)}\!\!\!\! (\bef(\bsfU_i^n) -\! \bef(\bsfU_j^n)) \SCAL \bc_{ij}
  + d_{ij}^{L,n} (\bsfU^n_j + \bsfU^n_i).
\end{align*}
Then, upon introducing the quantities
\begin{align}
\label{def_barstates}
\overline{\bsfU}_{ij}^{n+1} := \frac{1}{2}(\bsfU^{n}_i + \bsfU^{n}_j) 
-(\bef(\bsfU_j^n) - \bef(\bsfU_i^n))\cdot  \frac{\bc_{ij}}{2 d_{ij}^{L,n}},
\end{align}
with the the convention $\overline{\bsfU}_{ii}^{n+1} = \bsfU^{n}_i$,
and the notation
$\overline{\bsfU}_{ij}^{n+1} = (
\overline{\rho}_{ij}^{n+1},\overline{\bbm}_{ij}^{n+1},\overline{E}_{ij}^{n+1})\tr$
the low-order update $\bsfU_i^{L,n+1}$, can be represented as a convex
combination as follows:
\begin{align}
\label{def_dij_scheme_convex}
\bsfU_i^{L,n+1} = \bigg(1 - \sum_{i\not=j \in \calI(\Dom_i)} \frac{2 \dt d_{ij}^{L,n}}{m_i} \bigg) \overline{\bsfU}_{ii}^{n}
+  \sum_{i\not=j \in \calI(\Dom_i)} \bigg(\frac{2 \dt d_{ij}^{L,n}}{m_i} \bigg) \overline{\bsfU}_{ij}^{n+1},
\end{align}
under the appropriate CFL condition. Lemma~2.1 from
\citep{Guermond_Popov_Fast_Riemann_2016}, inspired by
\cite[\S5]{Perthame_Shu_1996} and
\cite[Eq.~(2.7)]{Nessyahu_Tadmor_1990}, is that the intermediate state
$\overline{\bsfU}_{ij}^{n+1}$ is a space average of the Riemann
solution $\bu(\bn_{ij},\bsfU_i^n,\bsfU_j^n)$; that is to say,
$\overline{\bsfU}_{ij}^{n+1}=
\int_{-\frac12}^{\frac12}\bu(\bn_{ij},\bsfU_i^n,\bsfU_j^n)(x,t)\dif x$
with $t:=\|\bc_{ij}\|_{\ell^2}/(2 d_{ij}^{L,n})$ provided
$t\lambda_{\max}(\bn_{ij},\bsfU_i^n,\bsfU_j^n)\le \frac12$.  (Let us emphasize tat the
time $t$ is related to the Riemann
problem~\eqref{one_d_riemann_with_n}, this time has noting to do with
that of the PDE~\eqref{Euler}.) Note that the definition of the
low-order graph viscosity \eqref{Def_of_dij_I} does imply
that $t\lambda_{\max}(\bn_{ij},\bsfU_i^n,\bsfU_j^n)\le \frac12$.  An
immediate consequence of this structure is that
$\overline{\bsfU}_{ij}^{n+1}$ has positive density, positive internal
energy and satisfies the following minimum principle on the specific
entropy:
$\Phi(\overline{\bsfU}_{ij}^{n+1}) \ge
\min(\Phi(\bsfU_i^n),\Phi(\bsfU_j^n))$.
Upon denoting
$d_{ii}^{L,n} := -\sum_{i\ne j\in \calI(\Dom_i)} d_{ij}^{L,n}$ for
brevity, another consequence of the above observation is the following
result.
\begin{theorem}[Local invariance/entropy inequality] \label{Thm:UL_is_invariant} 
Let $i\in\calI$. Assume \eqref{positivity_lumped_mass} and
 $1+2\dt \frac{d_{ii}^{L,n}}{m_i}\ge 0$.
\manuallabel{Item1:Thm:UL_is_invariant}{\textup{i}}
Let $s^{\min}_i=\min_{j\in\calI(\Dom_i)}\Phi(\bsfU_j^n)$,
then $\bsfU_i^{L,n+1} \in A_{s^{\min}_i}$.
\manuallabel{Item2:Thm:UL_is_invariant}{\textup{ii}}
 Let $(\eta:=\rho f(\Phi),\bq:=\bbm f(\Phi))$ be a generalized entropy pair
  for \eqref{Euler}. Then the following local
  entropy inequality holds:
\begin{multline*}
  \frac{m_i}{\dt} (\eta(\bsfU_i^{L,n+1}) - \eta(\bsfU_i^n))   
   + \int_\Dom \DIV\bigg(\sum_{j\in\calI(\Dom_i)}\bq(\bsfU_j^n)\varphi_j\bigg) \varphi_i \diff \bx \\
-\sum_{j\in \calI(\Dom_i)} d_{ij}^{L,n} \big(\eta(\bsfU_{j}^{n})-\eta(\bsfU_{j}^{n})\big) \ge 0. 
\end{multline*}
\end{theorem}

A practical interpretation of item \eqref{Item1:Thm:UL_is_invariant}
is that the low-order solution $\bsfU_i^{L,n+1}$ has positive density,
positive internal energy and satisfies the local minimum principle on
the specific entropy. Item~\eqref{Item2:Thm:UL_is_invariant} shows
that this solution is also entropy satisfying in some discrete sense.
This result is proved in \citep[Thm.~4.7]{Guer2016} in a more general
setting for any hyperbolic system with a convex entropy. Note that in
mathematical papers the entropies are generally assumed to be convex
whereas the physical generalized entropies $\rho f(\Phi)$ are concave
(this is just a matter of sign convention).  Hence, Theorem~4.7 in
\citep{Guer2016} can be applied with $-\rho f(\Phi)$.

\subsection{Smoothness-based approximation} \label{Sec:smoothness_based_viscosity} In this
section and the following one we introduce high-order approximation
techniques that will provide us with a provisional high-order solution
$\bsfU_j^{H,n+1}$. The method presented in this section
is easy to implement but is inherently only second-order accurate in space.
  
We introduce a technique to reduce the graph viscosity that is based
on a measure of the local smoothness of the solution in the spirit of
the finite volume literature (see
\eg\cite[Eq. (12)]{Jameson_aiaa_1981} and see the second formula in
the right column of page~1490 in \cite{Jameson_aiaa_2015}). Given a
scalar-valued function $g$ and its finite element interpolant
$g_h=\sum\sfG_i\varphi_i$, and denoting
$\epsilon_i = \epsilon \max_{j\in \calI(\Dom_i)} \sfG_j$ where
$\epsilon=10^{-\frac{16}{2}}$ in double precision arithmetic, we
define the smoothness indicator
\begin{equation}
  \alpha_i(g_h):= \frac{\left|\sum_{j\in\calI(S_i)} 
\beta_{ij}(\sfG_j-\sfG_i)\right|}{\max(\sum_{j\in\calI(S_i)} \beta_{ij}|\sfG_j-\sfG_i|,\epsilon_i)},
  \label{def_of_alpha_in}
\end{equation}
where the real numbers $\beta_{ij}$ are assumed to be positive.  One
can use the parameters $\beta_{ij}$ to make $\alpha_i=0$ if $g_h$ is
linear on the support of the shape function $\varphi_i$, this property
is called linearity-preserving (see \cite{Berger:2005:ASL} for a
review on linearity-preserving limiters in the finite volume
literature). Note that $\alpha_i\in [0,1]$ for all
$i\in\intset{1}{\Nglob}$ and $\alpha_i=1$ if $\sfG_i$ is a local
extrema of $g_h$.  Moreover, if the coefficients $\beta_{ij}$ are
defined so that $\alpha_i=0$ if $g_h$ is linear on $S_i$, then the
numerator of \eqref{def_of_alpha_in} behaves like
$h^2\|D^2 g(\bxi)\|_{\ell^2(\Real^{d\CROSS d})}$ at some point $\bxi$,
whereas the denominator behaves like
$h\|\GRAD g(\bzeta)\|_{\ell^2(\Real^d)}$ at some point
$\bzeta$. Therefore, we have
$ \alpha_i \approx h \|D^2 g(\bxi)\|_{\ell^2(\Real^{d\CROSS
    d})}/\|\GRAD g(\bzeta)\|_{\ell^2(\Real^d)}$,
that is to say $\alpha_i$ is of order $h$ in the regions where $g$ is
smooth and does not have a local extremum. In the computations
reported at the end of the paper we have taken $\beta_{ij}=1$. 

One choice for $g$ that we consider in some numerical tests reported
at the end of the paper consists using
$g(\bu) = \rho \Phi(\bu)= S(\bu)$, which is the mathematical entropy
(up to a sign). Other options consists of using generalized entropies
of the Euler equations, $g(\bu)=\rho f(\Phi(\bu))$.  In particular,
taking $f(s)=1$, gives $g(\bu)=\rho$, which is an extreme case of
generalized entropy; it is extreme in the sense that $-g(\bu)$ is
convex but not strictly convex. Note in passing that it is shown in
\cite{guermond_popov_positivity} that using $g(\bu)=\rho$ guarantees
positivity of the density.  Another option, which is somewhat similar
to \cite[Eq. (12)]{Jameson_aiaa_1981},
\cite[p.~1490]{Jameson_aiaa_2015}, consists of taking $g(\bu)=p$.
Note however that it might be better 
to take $p^{\frac1\gamma}$ to be entropy consistent, since
$p^{\frac1\gamma}$  is an extreme generalized entropy for polytropic gas as
shown in \cite[Eq.~(2.10a)]{Harten_1983}. Let us emphasize that strict
convexity of the entropy is not needed for the purpose of the present paper.

Let $\psi\in C^{0,1}([0,1];[0,1])$ be any positive function such that
$\psi(1)=1$.  The high-order smoothness-based graph viscosity is
defined by setting
\begin{equation} 
d_{ij}^{H,n} :=d_{ij}^{L,n}\max(\psi(\alpha^n_i),\psi(\alpha^n_j)), 
\quad d_{ii}^{H,n} := -\sum_{i\ne j\in \calI(\Dom_i)} d_{ij}^{H,n},
\label{def_dij_smooth}
\end{equation}
where $\alpha^n_i:=\alpha_i(g^n_h)$. A typical choice for $\psi$
consists of setting $\psi(\alpha)=\alpha^2$. Then the provisional
high-order approximation is computed as follows:
\begin{equation} 
\sum_{j\in \calI(\Dom_i)} \frac{m_{ij}}{\dt}(\bsfU_j^{H,n+1}-\bsfU_j^n)
  + \sum_{j\in \calI(\Dom_i)} \bef(\bsfU_j^n)\SCAL \bc_{ij}
  - d_{ij}^{H,n} (\bsfU^n_j-\bsfU^n_i)  = 0.
\label{def_high_order_scheme_smooth}
\end{equation}
Note that we use the consistent mass matrix to reduce dispersion error
since it is known that the use of the consistent mass matrix 
corrects the dominating dispersion error (at least for piecewise
linear approximation), see
\cite{FLD:FLD719,gresho1998incompressible,Guermond_pasquetti_2013,Thompson_2016}. The
beneficial effects of the consistent mass matrix are particularly
visible when solving problems with non-smooth solutions, see \eg
\citep[Fig.~5.5]{Guermond_pasquetti_2013}.

\subsection{Entropy viscosity commutator} \label{Sec:entropy_residual}
We now introduce a method that is formally high-order for any
polynomial degree, contrary the one introduced in
\S\ref{Sec:smoothness_based_viscosity}.  Our objective is to construct
a high-order method that is entropy consistent and close to be
invariant domain preserving. We do not want to rely on the yet to be
explained limiting process to enforce entropy consistency. We refer
the reader to Lemma~3.2, Lemma~4.4 and \S6.1 in
\cite{guermond_popov_second_order_2018} and \cite[\S5.1]{Guer2016} for
counter-examples of methods that are invariant domain preserving but
entropy violating. The heuristics we have in mind is that limitation should be
understood as a light polishing applied to a method that is already entropy
consistent and almost invariant domain preserving.  Following an idea
introduced in
\cite{Guermond_pasquetti_popov_2011,Guermond_Nazarov_Popov_Yong_2014},
we construct a high-order graph viscosity that is entropy consistent by estimating a
non-dimensional entropy residual. However, contrary to the techniques
introduced in
\citep{Guermond_pasquetti_popov_2011,Guermond_Nazarov_Popov_Yong_2014},
we do not want the time discretization to interfere with the
estimation of the residual, and we now propose a slightly different
approach. Given the current approximation $\bu_h^n$, we estimate the
next inviscid approximation by setting
$\bsfU_i^{G,n+1} := \bsfU_i^n - \frac{\dt}{m_i}
\sum_{j\in\calI(\Dom_i)} \bef(\bsfU_j^n)\bc_{ij}$.
Essentially $\bsfU_i^{G,n+1}$ is the Galerkin approximation of
$\bu(t^{n+1})$. Let $(\eta(\bv),\bF(\bv))$ be an entropy pair for
\eqref{Euler}. We estimate the entropy residual for the degree of
freedom $i$ by computing
$\frac{m_i}{\dt} (\bsfU_i^{G,n+1}-\bsfU_i^n)\SCAL\eta'(\bsfU_i^n) +
\sum_{j\in\calI(\Dom_i)} \bF(\bsfU_j^n)\bc_{ij}$.
But using the definition of $\bsfU_i^{G,n+1}$, this is equivalent to
computing
$\sum_{j\in\calI(\Dom_i)} ( \bF(\bsfU_j^n) - \eta'(\bsfU_i^n)\tr
\bef(\bsfU_j^n))\bc_{ij}$.
Then, upon setting
$\eta_i^{\min} := \min_{j\in\calI(\Dom_i)} \eta(\bsfU_j^n)$,
$\eta_i^{\max} := \max_{j\in\calI(\Dom_i)} \eta(\bsfU_j^n)$,
$\epsilon_i:=\epsilon \max(|\eta_i^{\max}|,|\eta_i^{\min}|)$, we adopt
the following definition
\begin{equation}
  R_i^n = \frac{1}{\max(\eta_i^{\max} -\eta_i^{\min},\epsilon_i)} \sum_{j\in\calI(\Dom_i)} 
(\bF(\bsfU_j^n) -
  \eta'(\bsfU_i^n)\tr \bef(\bsfU_j^n)) \SCAL \bc_{ij}, \label{def_of_Rin}
\end{equation}
where $\epsilon$ a tiny number that avoids division by zero when the
entropy is constant over $\Dom_i$.  In practice we take
$\epsilon=10^{-\frac{16}{2}}$ in double precision arithmetic. Note
that $R_i^n$ can be interpreted as a commutator.  More specifically
$R_i^n$ can be rewritten as follows
$\frac{1}{\max(\eta_i^{\max} -\eta_i^{\min},\epsilon_i)} \int_\Dom
(\DIV(\Pi_h\bF(\bu_h^n)) - \eta'(\bsfU_i^n)\tr \DIV(\Pi_h(\bef(\bu_h^n))
))\varphi_i \diff \bx$
where, up to two slight abuses of notation,
$\Pi_h : C^0(\Dom)\to P(\calT_h)$ is the Lagrange interpolation
operator.  Note in passing that $R_i^n=0$ in the hypothetical case
that $\eta:\Real^{d+2}\to \Real$ is linear.

The high-order graph viscosity (or entropy
viscosity, EV) is defined by
\begin{equation}
d_{ij}^{H,n} = \min(d_{ij}^{L,n},\max(|R_i^n|,|R_j^n|)),  
\quad d_{ii}^{H,n} := -\sum_{i\ne j\in \calI(\Dom_i)} d_{ij}^{H,n},
\label{def_dij_EV}
\end{equation}
and the provisional high-order approximation is computed as follows:
\begin{equation}
\label{def_high_order_scheme_EV}
  \sum_{j\in \calI(\Dom_i)} \frac{m_{ij}}{\dt}(\bsfU_j^{H,n+1}-\bsfU_j^n)
  + \sum_{j\in \calI(\Dom_i)} \bef(\bsfU_j^n)\SCAL \bc_{ij}
  - d_{ij}^{H,n} (\bsfU^n_j-\bsfU^n_i)  = 0.
\end{equation}
Note again that we use the consistent mass matrix to reduce dispersion
errors as explained in \S\ref{Sec:smoothness_based_viscosity} .

\begin{remark}[Scaling of $R_i^n$]
  Let us now convince ourselves that $R_i^n$ is at least one order
  smaller (in term of mesh size) than $d_{ij}^{L,n}$. Let use denote
  $F''_{\max} $ and $f''_{\max} $ the maximum over the convex hull
  $\text{conv}(\bsfU_j^n)_{j\in \calI(\Dom_i)}$ of the matrix norm
  (say the norm induced by the Euclidean norm in $\Real^{d+2}$) of the
  Hessians $D^2\bF$ and $D^2 f$. Then, denoting by $N_i^n$ the
  numerator in \eqref{def_of_Rin}, and recalling that
  $D\bF(\bsfU) = \eta'(\bsfU)\tr D\bef(\bsfU)$, we have
\begin{align*}
\|N_i^n\|_{\ell^2} &= \|\sum_{j\in\calI(\Dom_i)}(\bF(\bsfU_j^n) -\bF(\bsfU_i^n) -
  \eta'(\bsfU_i^n)\tr (\bef(\bsfU_j^n)-\bef(\bsfU_i^n)\SCAL \bc_{ij}\|_{\ell^2} \\
& \le \frac12 (F''_{\max}  + \eta'(\bsfU_i^n) f''_{\max}  ) 
\max_{j\in\calI(\Dom_i)}\|\bc_{ij}\|_{\ell^2}\sum_{j\in\calI(\Dom_i)}
\|\bsfU_j^n-\bsfU_i^n\|_{\ell^2}^2
.
\end{align*}
Assuming that $\eta'$ is not zero over $\Dom_i$ and denoting by
$\eta'_{\min}$ the minimum of $\|\eta'\|_{\ell^2}$ over
$\text{conv}(\bsfU_j^n)_{j\in \calI(\Dom_i)}$, the quantity
$\eta'_{\min}
\sum_{j\in\calI(\Dom_i)}\|\bsfU_j^n-\bsfU_i^n\|_{\ell^2}$
is a lower bound for the denominator in \eqref{def_of_Rin}.
The conclusion follows readily.
\end{remark}

\begin{remark}[Choice of high-order graph viscosity] 
  One advantage we see in the entropy viscosity \eqref{def_dij_EV}
  over the smoothness-based viscosity \eqref{def_dij_smooth} is that,
  in addition to being consistent for any polynomial degree, it is
  also consistent with at least one entropy inequality. That is to say
  the viscosity is large when the entropy production is large and it
  is small otherwise. In any case, we have observed that
  \eqref{def_dij_EV} always gives a scheme that is more robust than
  \eqref{def_dij_smooth} albeit being slightly more oscillatory. We
  refer the reader to \cite[\S6.5]{guermond_popov_second_order_2018}
  where this issue is discussed in detail.
\end{remark}

\begin{remark}[Entropy]
  We have found in our numerical experiments that using
  $p^{\frac{1}{\gamma}}$ for polytropic gases is a very good choice to
  construct the entropy residual since it minimizes dissipation across
  contact discontinuities.  Recall that $p^{\frac1\gamma}$ is indeed
  an entropy in the case of polytropic gases. All the numerical tests
  reported at the end of the paper are done with this entropy.
\end{remark}
    
\begin{remark}[Entropy ansatz]
  In realistic applications the equation of state is often tabulated
  and the entropy $\Phi(\bsfU)$ may be either unavailable or costly to
  estimate. We have found that the pressure may be used as an ansatz for the entropy;
one can then replace \eqref{def_of_Rin} by
  $R_i^n = \frac{1}{\max(\sfP_i^{\max} -\sfP_i^{\min},\epsilon_i)}
  \sum_{j\in\calI(\Dom_i)} \bsfV_i^n\SCAL \bc_{ij}(\sfP_j^n -
  D\sfP_i^n \SCAL \bsfU_j^n )$
  with $\epsilon_i = \epsilon \max(|\sfP_i^{\max}|, |\sfP_i^{\min}|)$.
\end{remark}

\section{Quasiconcavity-based limitation}
\label{Sec:convexity_based_limitation}
In this section we discuss the bounds we want the numerical solution
$\bsfU_i^{n+1}$ to satisfy, and we develop a novel limiting technique
that is convexity-based and does not invoke arguments like
linearization, worst-case scenario estimates, a posteriori fixes, or
auxiliary discontinuous spaces as often done in the literature. This
technique takes its roots in \cite{Khobalatte_Perthame_1994},
\cite{Perthame_Youchun_1994}, and \cite{Perthame_Shu_1996}. We also
refer to \cite{Zhang_Shu_2010,Zhang_Shu_2012}, \cite{Jiang_Liu_2017}
for extensions in the context of the Discontinuous Galerkin
approximation.

\subsection{Bounds and quasiconcavity} \label{Sec:bounds}
Since the high-order update $\bsfU_i^{H,n+1}$ (using either
\eqref{def_high_order_scheme_smooth} or
\eqref{def_high_order_scheme_EV}) is not guaranteed to be oscillation
free and to preserve physical bounds, some form of
limitation must applied. The question is now the following: what should be limited
and how?  Whichever representation is chosen for the dependent
variable (conservative, primitive, or characteristic variables), the
Euler equations are not known to satisfy any maximum or minimum
principle, with the exception of the minimal principle on the specific
entropy. Despite this fundamental negative result and with varying
levels of success, a number of techniques have been
proposed over the years in the finite element literature to enforce
some kinds of discrete maximum principles (see for instance
\cite{Boris_books_JCP_1973,Zalesak_1979,Lohner1987, Kuzmin2005,
  Zalesak2005}, \cite{Lohm2016}). Some of these limiting techniques
enforce properties that are not necessarily satisfied by the Euler
equations, or in the best case scenario, satisfied by the first-order
method of choice (usually a Lax-Friedrichs-like first-order scheme).

In the present paper, we take a different point of view.
In addition to the local minimum principle on the specific entropy,
the strategy that we propose consists of enforcing bounds that are naturally
satisfied by the low-order solution. More precisely,
let us set
\begin{align}
\label{rhominmaxdef}
\rho_i^{\text{min}}&:= \min_{j\in \calI(\Dom_i)} (\overline{\rho}_{ij}^{n+1},\rho_{j}^{n}) ,\qquad
\rho_i^{\text{max}} := \max_{j\in \calI(\Dom_i)} (\overline{\rho}_{ij}^{n+1},\rho_{j}^{n}), \\
E_i^{\text{min}}&:= \min_{j\in \calI(\Dom_i)} (\overline{E}_{ij}^{n+1},E_{j}^{n}),\qquad
E_i^{\text{max}} :=\max_{j\in \calI(\Dom_i)} (\overline{E}_{ij}^{n+1},E_{j}^{n}),\label{Eminmaxdef}\\
s^{\min}_i&:= \min_{j\in\calI(\Dom_i)}\Phi(\bsfU_j^n). \label{smindef}
\end{align}
We have already established in \S\ref{Sec:intermediate_limiting_states} that
$\rho_i^{\text{min}} \leq \rho_i^{L,n+1} \leq \rho_i^{\text{max}}$,
$E_i^{\text{min}} \leq E_j^{L,n+1} \leq E_i^{\text{max}}$ and
$s^{\min}_i\le \Phi(\bsfU_i^{L,n+1})$.  In the next section we are
going to modify the graph viscosity so that the resulting high-order update
$\bsfU_i^{n+1}$ satisfies
$\rho_i^{\text{min}} \leq \rho_i^{n+1} \leq \rho_i^{\text{max}}$ and
$s^{\min}_i\le \Phi(\bsfU_i^{n+1})$ (and possibly
$E_i^{\text{min}} \leq E_j^{n+1} \leq E_i^{\text{max}}$ if one wishes).

In general, the equation for the specific entropy may not be
explicitly available and therefore limiting the specific entropy may
not be possible. An alternative strategy consists of limiting the
internal energy $\rho e$. Using the Frechet derivative notation, a
straightforward computation shows that
$D^2(\rho e(\bu))((\varrho,\bp,a),(\varrho,\bp,a)) = -
\frac{1}{\rho}(\frac{\varrho}{\rho}\bbm - \bp)^2$
for all directions $(\varrho,\bp,a)\tr\in \Real^{d+2}$ and all points
$\bu=(\rho,\bbm,E)\tr\in \Real^{d+2}$, thereby showing that the internal
energy is concave with respect to the conservative variables
irrespective of the equation of state. Hence, the concavity of
$(\rho e)$ along with the convex combination
\eqref{def_dij_scheme_convex} implies that the low-order solution
$\bsfU^{L,n+1}_i$ satisfies the following discrete minimum principle
\begin{equation}\label{minimum_internal_energy}
  (\rho e)(\bsfU^{L,n+1}_i) \ge \varepsilon_i^{\min}:= \min\Big(\min_{j\in\calI(\Dom_i)}(\rho e)(\bsfU^{n}_j), 
  \min_{j\in\calI(\Dom_i)}(\rho e)(\overline{\bsfU}^{n+1}_{ij})\Big).
\end{equation}

In order to unify into one single framework all the bounds that we
want to enforce, we are going to rely on the notion of quasiconcavity,
which definition we now recall.
\begin{definition}[Quasiconcavity]
  Given a convex set $\calA\subset \Real^m$, we say that a function
  $\Psi:\calA \to \Real$ is quasiconcave if every {\em upper} level
  set of $\Psi$ is {\em convex}; that is, the set
  $L_\lambda(\Psi) := \{\bsfU\in\calA \st \Psi(\bsfU) \ge \lambda \}$
  is convex for any $\lambda\in \Real$ in the range of $\Psi$.
\end{definition}
Note in passing that concavity implies quasiconcavity.  We are going
to use the above definition in the following three settings:
\textup{(i)} $\calA=\Real^{d+2}$ and
$\Psi(\bsfU) =\rho - \rho_i^{\min}$ or
$\Psi(\bsfU) =\rho_i^{\max}-\rho$.  Note that in both cases the upper
level sets are half spaces (\ie these sets are obviously convex); 
\textup{(ii)} $\calA=\{\bsfU:=(\rho,\bbm,E) \st \rho>0\}$ and
$\Psi(\bsfU) = (\rho e) - \varepsilon_i^{\min}$. We have shown above that 
that $(\rho e)(\bsfU)$ is concave provided $\rho>0$ (the Hessian
of $\rho e$ is nonpositive), then it follows that
$\{\bsfU:=(\rho,\bbm,E) \st \rho>0, e>0\}$ is convex;
\textup{(iii)} $\calA=\{\bsfU:=(\rho,\bbm,E) \st \rho>0, e>0\}$,
$\Psi(\bsfU) = \Phi(\bsfU) - s_i^{\min}$. The quasiconvexity of
$\Psi:\calA \to \Real$ is proved in \cite[Thm.~8.2.2]{Serre_2000_II}.

\begin{remark}[Concavity vs. quasiconcavity]
Note that the two sets 
\[
\{(\rho,\bbm,E) \st \rho>0, e>0, s\ge r\},\qquad
\{(\rho,\bbm,E) \st \rho>0, \rho e>0, \rho (s -r)\ge 0\}
\]
are identical. In the first case, quasiconcavity is invoked to prove
that the upper level sets are convex, whereas in the second case one
just has to rely on concavity since the three functions $\rho$,
$\rho e(\bu)$, and $\rho (\Phi(\bu)-r)$ are concave. It is easier to
impose concave (or convex) constraints than quasiconcave ones. More
precisely: in practice it is simpler to apply Newton's method on a
concave function than on a quasiconcave function; in the first case
Newton's method is guaranteed to converge independently of the initial
guess, whereas it may not in the second case.
\end{remark}

\subsection{An abstract limiting scheme} 
Simple linear constraints such as
$ \overline{\rho}_i^{\text{min}} \leq \rho_i^{n+1} \leq
\overline{\rho}_i^{\text{max}}$
and
$\overline{E}_i^{\text{min}} \leq E_j^{n+1} \leq
\overline{E}_i^{\text{max}}$,
can be easily enforced by using the Flux Transport Corrected paradigm
of \cite{Zalesak_1979} (see also \cite{Boris_books_JCP_1973}).
However, to the best of our knowledge, the Zalezak's grouping
methodology cannot be (easily) extended to handle general convex
constraints like the minimum principle on the specific entropy without losing second-order
accuracy. We introduce in this section a methodology that does exactly
that.

We start as in the FCT methodology by estimating the difference $\bsfU_j^{H,n+1}-\bsfU_j^{L,n+1}$.
Subtracting \eqref{def_high_order_scheme_smooth} (or
\eqref{def_high_order_scheme_EV}) from \eqref{def_dij_scheme} we
obtain that the low-order and the provisional high-order solutions
satisfy the following identity
\[
\sum_{i\in \calI(\Dom_i)} \!\!m_{ij}(\bsfU_j^{H,n+1}-\bsfU_j^{n}) 
-\dt d_{ij}^{H,n} (\bsfU_j^{n}-\bsfU_i^{n}) = m_{i}(\bsfU_i^{L,n+1}-\bsfU_i^{n}) 
-\dt d_{ij}^{L,n} (\bsfU_j^{n}-\bsfU_i^{n}).
\]
This equality is rewritten in the following form better suited for
post-processing:
\[
m_i (\bsfU_j^{H,n+1}-\bsfU_j^{L,n+1}) 
= \sum_{i\in \calI(\Dom_i)}\Delta_{ij}(\bsfU_j^{H,n+1}-\bsfU_j^{n})
+ \dt (d_{ij}^{H,n} -d_{ij}^{L,n}) (\bsfU_j^{n}-\bsfU_i^{n}),
\]
where we have set $\Delta_{ij}:= m_i\delta_{ij}-m_{ij}$.
The above identity can be rewritten as follows:
\begin{equation}
\label{RegFCT}
\left\{\begin{aligned}
&m_i (\bsfU^{H,n+1}-\bsfU^{L,n+1})= \sum_{j\in\calI(\Dom_i)} \bA_{ij}^n \\
&\bA_{ij}^n:=\Delta_{ij}\big(\bsfU^{H,n+1}_{j}-\bsfU^{n}_j - (\bsfU^{H,n+1}_{i}-\bsfU^{n}_i)\big)
-\dt (d^{H,n}_{ij}-d^{L,n}_{ij})(\bsfU^{n}_j-\bsfU^{n}_i),
\end{aligned}\right.
\end{equation}
where we used $\sum_{j\in\calI(\Dom_i)} \Delta_{ij} =0$.
Observe that the matrix $\bA^n$ is skew-symmetric; the immediate
consequence is that
$\sum_{i\in \calI}m_i \bsfU_j^{H,n+1} = \sum_{i\in \calI}m_i
\bsfU_j^{L,n+1}$,
\ie the total mass of the provisional high-order solution is the same
at that of the low-order solution.  

The next step consists of introducing symmetric limiting parameters
$\limiter_{ij}= \limiter_{ji}\in [0,1]$ and estimating $\limiter_{ij}$ so
that the new quantity
$\bsfU^{n+1} = \bsfU^{L,n+1} +\frac{1}{m_i} \sum_{j\in\calI(\Dom_i)}
\limiter_{ij}\bA_{ij}^n$
satisfies the expected bounds.  Note again that the skew-symmetry of
$\bA^n$ together with the symmetry of the limiter implies that
$\sum_{i\in \calI}m_i \bsfU_j^{n+1} = \sum_{i\in \calI}m_i
\bsfU_j^{L,n+1}$
for any choice of limiter $\limiter_{ij}$, \ie the limiting process is
conservative.  Using the notation introduced at the end of
\S\ref{Sec:bounds}, we seek $\limiter_{ij}$ so that
$\Psi(\bsfU^{n+1})\ge 0$.
%We have not been able to modify the FCT
%algorithm to achieve this task without destroying the
%high-order accuracy of $\bsfU^{H,n+1}$.

We now depart form the FCT algorithm as described in \citep{Zalesak_1979} by introducing
$\lambda_j := \frac{1}{\text{card}(\calI(\Dom_i))-1}$, $j\in\calI(\Dom_i)\setminus\{i\}$, and rewriting \eqref{RegFCT} as follows
\begin{align}
\label{convFCT}
\bsfU_i^{n+1} = \sum_{j \in \calI(\Dom_i)\backslash\{i\} } \lambda_j (\bsfU_i^{L,n+1} +\limiter_{ij}\bsfP_{ij}),
\qquad \text{with} \qquad \bsfP_{ij} := \frac{1}{m_i \lambda_j} \bsfA_{ij}^n.
\end{align}
Note that $\bsfU_i^{n+1}=\bsfU_i^{L,n+1}$ if $\limiter_{ij}=0$ and
$\bsfU_i^{n+1}=\bsfU_i^{H,n+1}$ if $\limiter_{ij}=1$.  The following
lemma  is the driving force of
the limiting technique that we propose.
\begin{lemma} \label{Lem:quasiconcave} Let
  $\Psi(\bu):\calA \rightarrow \Real$ be a quasiconcave
  function. Assume that the limiting parameters
  $\limiter_{ij} \in [0,1]$ are such that be such that
  $\Psi (\bsfU_i^{L,n+1} + \limiter_{ij} \bsfP_{ij}) \geq 0$,
  $i\ne j \in \calI(\Dom_i)$, then the following inequality holds
  true:
\begin{align*}
\Psi \bigg(\sum_{j \in \calI(\Dom_i)\setminus\{i\}} 
\lambda_j (\bsfU_i^{L,n+1} + \limiter_{ij} \bsfP_{ij}) \bigg) \geq 0 .
\end{align*}
\end{lemma}

\begin{proof} 
  Let $L_0=\{\bsfU \in\calA \st \psi(\bsfU)\ge 0 \}$. By
  definition all the limited states
  $\bsfU_i^{L,n+1} + \limiter_{ij} \bsfP_{ij}$ are in $L_0$ for all $i\ne j \in \calI(\Dom_i)$.  Since
  $\Psi$ is quasiconcave, the upper level set $L_0$ is convex. As a result, the
  convex combination
  $\sum_{j \in \calI(\Dom_i)\setminus\{i\}} \lambda_j (\bsfU_i^{L,n+1}
  + \limiter_{ij} \bsfP_{ij})$
  is in $L_0$, \ie
  $\Psi \big(\sum_{j \in \calI(\Dom_i)\setminus\{i\}} \lambda_j
  (\bsfU_i^{L,n+1} + \limiter_{ij} \bsfP_{ij}) \big) \geq 0$,
  which concludes the proof.
\end{proof}

\begin{lemma} \label{Lem:compute_lij} Let $\limiter^i_j$ be defined by
\begin{equation}\label{Eq:Lem:compute_lij}
\limiter^i_j=
\begin{cases}
1 & \text{if $\Psi(\bsfU_i^{L,n+1} + \bsfP_{ij})\ge 0$} \\
\max\{\limiter \in [0,1] \st \Psi(\bsfU_i^{L,n+1} +
  \limiter \bsfP_{ij})\ge 0\} & \text{otherwise},
\end{cases}
\end{equation}
for every $i\in\calI$ and $j\in \calI(\Dom_i)$. The following two statements hold true:
\textup{(i)} $\Psi(\bsfU_i^{L,n+1} + \limiter \bsfP_{ij})\ge 0$ for every $\limiter \in [0,\limiter^i_j]$;
\textup{(ii)} In particular, setting
  $\limiter_{ij} = \min(\limiter^i_j, \limiter^j_i)$, we have
  $\Psi(\bsfU_i^{L,n+1} + \ell_{ij}\bsfP_{ij})\ge 0$ and
  $\limiter_{ij}=\limiter_{ji}$ for every $i\in\calI$ and
  $j\in \calI(\Dom_i)$.
\end{lemma}
\begin{proof} \textup{(i)}
  First, if $\Psi(\bsfU_i^{L,n+1} + \bsfP_{ij})\ge 0$ we observe that
  $\Psi(\bsfU_i^{L,n+1} + \limiter \bsfP_{ij})\ge 0$ for any
  $\limiter\in [0,1]$ because $\bsfU_i^{L,n+1}\in L_0(\Psi)$,
  $\bsfU_i^{L,n+1} + \bsfP_{ij}\in L_0(\Psi)$ and $L_0(\Psi)$ is
  convex. Second, if $\Psi(\bsfU_i^{L,n+1} + \bsfP_{ij})< 0$, we
  observe that $\limiter^i_j$ is uniquely defined and for any
  $\limiter\in [0,\limiter^i_j]$ we have
  $\Psi(\bsfU_i^{L,n+1} + \limiter \bsfP_{ij})\ge 0$ because
  $\bsfU_i^{L,n+1}\in L_0(\Psi)$,
  $\bsfU_i^{L,n+1} + \limiter^i_j \bsfP_{ij}\in L_0(\Psi)$ and
  $L_0(\Psi)$ is convex. \textup{(ii)}  Since
  $\limiter_{ij} = \min(\limiter^i_j, \limiter^j_i)\le \limiter^i_j$,
  the above construction implies that
  $\Psi(\bsfU_i^{L,n+1} + \ell_{ij}\bsfP_{ij})\ge 0$. Note finally
  that
  $\limiter_{ij} = \min(\limiter^i_j, \limiter^j_i) = \limiter_{ji}$.
\end{proof}

\begin{remark}[Extension to general hyperbolic systems]
  Notice that Lemma~\ref{Lem:quasiconcave} and
  Lemma~\ref{Lem:compute_lij} are not specific to the Euler equations.
  These results can be used to limit solutions of arbitrary hyperbolic
  systems where the invariant domain is described by
  quasiconcave constraints.
\end{remark}

\subsection{Application to the Euler equations}
We explain in this section how to use Lemma~\ref{Lem:quasiconcave} and
Lemma~\ref{Lem:compute_lij} to enforce the quasiconcave constraints
described in \S\ref{Sec:bounds}. 
The algorithm goes as follows:

(i) Given the state $\bsfU^n$, which we assume to be admissible, we
compute $\bsfU^{L,n+1}$ and $\bsfU^{H,n+1}$ as explained in
\S\ref{Sec:low_order} and either
\S\ref{Sec:smoothness_based_viscosity} or \ref{Sec:entropy_residual};

(ii) The density is limited by invoking Lemma~\ref{Lem:quasiconcave}
and Lemma~\ref{Lem:compute_lij} and the bounds described in
\S\ref{Sec:bounds} to enforce the quasiconcave constraints
$\Psi(\bsfU) = \rho-\rho_i^{\min}\ge 0$ and
$\Psi(\bsfU) = \rho_i^{\max}-\rho\ge 0$. The resulting limiter is
denoted by $\limiter_{ij}^\rho$ and details of the computation of
$\limiter_{ij}^\rho$ are given in \S\ref{Sec:density};

(iii) The internal energy
$\rho e:= E - \frac{\bbm^2}{2\rho}$ is limited by invoking
Lemma~\ref{Lem:quasiconcave} and Lemma~\ref{Lem:compute_lij} to
enforce the quasiconcave constraint
$\Psi(\bsfU) =E - \frac{\bbm^2}{2\rho} - \varepsilon_i^{\min} \ge 0$.
The corresponding limiter is denoted
$\limiter_{ij}^e\le \limiter_{ij}^\rho$ and details of the computation
of $\limiter_{ij}^e$ are given in \S\ref{Sec:rho_e};

(iv) The minimum principle on the specific entropy is enforced by
using $\Psi(\bsfU) = \Phi(\bsfU) - s_i^{\min}$. The details on the computation
of the corresponding  limiter $\limiter_{ij}^s\le \limiter_{ij}^e$ are 
given in \S\ref{Sec:specific_entropy}.

(v) Finally, upon setting $\limiter_{ij}:= \limiter_{ij}^s$, the update
$\bsfU^{n+1}$ is computed by setting
$\bsfU^{n+1}=\bsfU^{L,n+1}+ \frac{1}{m_i}\sum_{j\in\calI(\Dom_i)}
\limiter_{ij} \bA_{ij}^n$.
This type of limitation can be iterated a few times by observing that
\[
\bsfU^{H,n+1}=\bsfU^{L,n+1}+ \frac{1}{m_i}\sum_{j\in\calI(\Dom_i)} \limiter_{ij} \bA_{ij}^n
+ \frac{1}{m_i}\sum_{j\in\calI(\Dom_i)}(1- \limiter_{ij}) \bA_{ij}^n.
\]
Then setting $\bsfU^{(0)}:=\bsfU^{L,n+1}$ and
$ \bA_{ij}^{(0)}:= \bA_{ij}^n$ the iterative algorithm proceeds as
shown in Algorithm~\ref{iterative_algorithm}. In the numerical
simulations reported at the end of the paper we have taken
$ k_{\max}=1$.
\begin{algorithm}[H]
\renewcommand{\algorithmicrequire}{\textbf{Input:}}
\renewcommand{\algorithmicensure}{\textbf{Output:}}
\caption{Iterative limiting}
\label{iterative_algorithm}
\begin{algorithmic}[1]
\Require $\bsfU^{L,n+1}$, $\bA^n$, and $k_{\max}$
\Ensure $\bsfU^{n+1}$
\State Set $\bsfU^{(0)}:=\bsfU^{L,n+1}$ and $ \bA^{(0)}:= \bA^n$
\For{$k=0 \textbf{ to }  k_{\max}-1$}
\State Compute limiter matrix $\limiter^{(k)}$
\State Update $\bsfU^{(k+1)}=\bsfU^{(k)}+ \frac{1}{m_i}\sum_{j\in\calI(\Dom_i)} \limiter^{(k)}_{ij} \bA_{ij}^{(k)}$
\State Update $\bA_{ij}^{(k+1)} = (1-\limiter_{ij}^{(k)})\bA_{ij}^{(k)}$
\EndFor
\State   $\bsfU^{n+1}= \bsfU^{(k_{\max})}$
\end{algorithmic}
\end{algorithm}

\begin{remark}[Other quantities]
  As observed in Section~\ref{Sec:bounds} it is also possible to
  impose additional limiting based on quasiconcave constraints. For
  example, one could limit the total energy from below and from
  above. Numerical experiments reveal that this extra limitation does
  not improve the performance of the scheme. All the tests reported in
  \S\ref{Sec:Numerical_illustrations} are done by limiting the
  density, the internal energy and the specific entropy as described
  above. We have found that limiting the internal energy delivers
  second-order accuracy in the maximum norm, but has a tendency to
  over-dissipate contact discontinuities. Note that limiting the
  specific entropy amounts in effect to limit the internal energy.
\end{remark}

\begin{remark}[Equation of state]
  So far, everything we have described is independent of the equation of
  state.
\end{remark}

\subsection{Limitation on the density} \label{Sec:density} The limitation on the density
as specified by \eqref{Eq:Lem:compute_lij}
proceeds as follows. To avoid divisions by zero, we introduce the
small parameter $\epsilon:= 10^{-14}$ set
$\epsilon_i=\epsilon\rho_i^{\max}$ for all $i\in\calI$. Let us denote by $P_{ij}^{\rho}$ the
$\rho$-component of $\bP_{ij}$, and let us set
\begin{equation}
\limiter^{i,\rho}_j=
\begin{cases}
  \min(\frac{|\rho_i^{\min}-\rho_i^{L,n+1}|}{|P_{ij}^{\rho}|+\epsilon_i},1)
  & \text{if $\rho_i^{L,n+1}+P_{ij}^{\rho} < \rho_i^{\min}$} \\
  1 & \text{if $ \rho_i^{\min}\le \rho_i^{L,n+1}+P_{ij}^{\rho} \le \rho_i^{\max}$}\\
 \min(\frac{|\rho_i^{\max}-\rho_i^{L,n+1}|}{|P_{ij}^{\rho}|+\epsilon_i},1)
  & \text{if $\rho_i^{\max}<\rho_i^{L,n+1}+P_{ij}^{\rho}$.}
\end{cases} \label{def_of_lrho}
\end{equation}
Upon setting $\Psi_+(\bsfU) = \rho - \rho_i^{\min}$, and
$\Psi_-(\bsfU) =\rho_i^{\max} - \rho$, we have the following result
whose proof is left to the reader.
\begin{lemma}
  The definition \eqref{def_of_lrho} implies that
  $\Psi_\pm(\bsfU_i^{L,n+1} + \limiter \bsfP_{ij} )\ge 0$ for all
  $\ell\in [0, \limiter_j^{i,\rho}]$.
\end{lemma}

\begin{remark}[Covolume EOS]
  In the case of the covolume equation of state,
  $p(1-b \rho) = (\gamma-1)\rho e$, it is known that
  $\calA=\{(\rho,\bbm,E) \st \rho>0, e>0, s\ge s^{\min}, b\rho<1\}$ is an
  invariant domain, see
  \cite[Prop.~A.1]{Guermond_Popov_Fast_Riemann_2016}. The above method
  can be used to enforce the additional affine constraint $1-b\rho >0$.
\end{remark}

\begin{remark}[Total energy]
  The limitation on the total energy can be done exactly as for the
  density. Let us emphasize though that we have not found this operation to be 
useful and it is not done in the numerical tests reported at the end of the paper.
\end{remark}

\subsection{Limitation on the internal energy $\rho e$} \label{Sec:rho_e}
In this section we explain how to compute the limiter to enforce the
local minimum principle on the internal energy $\rho e$ as stated
in \eqref{minimum_internal_energy}.

Upon setting $\Psi(\bsfU) := (\rho e)(\bsfU) - \varepsilon_i^{\min}$
with $\bsfU:=(\rho,\bbm,E)$, and by virtue of
Lemma~\ref{Lem:quasiconcave} and Lemma~\ref{Lem:compute_lij}, we have
to estimate $\limiter_j^{i,e}\in [0,\limiter_j^{i,\rho}]$ so that
$\Psi(\bsfU_i^{L,n+1} + \limiter \bsfP_{ij}) \ge 0$ for all
$\limiter\in [0, \limiter_j^{i,e}]$. We define the auxiliary function $\psi:\{\bsfU\st \rho>0\}\to \Real$
\[
\psi(\bsfU) :=(\rho^2 e)(\bsfU) - \varepsilon_i^{\min}\rho = \rho E - \frac12
\bbm^2 - \varepsilon_i^{\min}\rho.
\]
Then the above problem is equivalent seeking
$\limiter_j^{i,e}\in [0,\limiter_j^{i,\rho}]$ so that
$\psi(\bsfU_i^{L,n+1} + \limiter \bsfP_{ij})\ge 0$ for all
$\limiter\in [0, \limiter_j^{i,e}]$. The key observation is that
now $\psi$ is a quadratic functional with 
\[
D\psi(\bsfU) = \begin{pmatrix}E- \varepsilon_i^{\min}\\ -\bbm\\ \rho\end{pmatrix},\qquad
D^2\psi(\bsfU) =\begin{pmatrix}0 & \bzero\tr & 1 \\ \bzero & -\polI & \bzero \\ 1 & \bzero\tr & 0\end{pmatrix}.
\]
Then upon setting $a:=\frac12  \bsfP_{ij}\tr D^2\psi \bsfP_{ij}$, $b:= D\psi(\bsfU_i^{L,n+1})\SCAL \bsfP_{ij}$ and $c:=\psi(\bsfU_i^{L,n+1})$, we have
\[
\psi(\bsfU_i^{L,n+1} + t \bsfP_{ij}) = a t^2 + b t + c.
\]
Let $t_0$ be the smallest positive root of the equation $a t^2 + b t + c=0$,
with the convention that $t_0=1$ if the equation has no positive root.
Then we choose $\limiter_j^{i,e}$ to be such that
\begin{equation}
\limiter_j^{i,e} = \min(t_0,\limiter_j^{i,\rho}). \label{def_of_le}
\end{equation}

\begin{lemma}
The definition \eqref{def_of_le} implies that $\Psi(\bsfU_i^{L,n+1} + \limiter \bsfP_{ij} )\ge 0$ for all
$\limiter\in [0, \limiter_j^{i,e}]$.
\end{lemma}
\begin{proof}
  If there is no positive root to the equation $a t^2 + b t + c =0$ and
  since we have established that $c= \psi(\bsfU_i^{L,n+1})\ge 0$ (see
  \eqref{minimum_internal_energy}), we have $a t^2 + b t + c\ge 0$ for
  all $t\ge 0$; that is,
  $\Psi(\bsfU_i^{L,n+1} + \limiter \bsfP_{ij} )\ge 0$ for all
  $\ell\ge 0$, and in particular this is true for all
  $\limiter\in [0, \limiter_j^{i,e}]$.  Otherwise, if there is at
  least one positive root to the equation $a t^2 + b t + c=0$, then
  denoting by $t_0$ the smallest positive root, we have
  $a t^2 + b t + c\ge 0$ for all $t\in[0, t_0]$ (if not, there would
  exist $t_1\in (0,t_0)$ s.t.  $a t_1^2 + b t_1 + c< 0$ and the
  intermediate value theorem would imply the existence a root
  $t^*\in(0,t_1)$ which contradict that $t_0$ is the smallest positive
  root).  This implies that
  $\Psi(\bsfU_i^{L,n+1} + \limiter \bsfP_{ij} )\ge 0$ for all
  $\limiter\in[0,t_0]$, and  in particular this is true for all
  $\limiter\in [0, \limiter_j^{i,e}]$ owing to \eqref{def_of_le}.
\end{proof}

\begin{remark}[Equation of state]
  Observe that the proposed limitation on $\rho e$ is independent of
  the equation of state.  
\end{remark}

\subsection{Minimum principle on the specific
  entropy} \label{Sec:specific_entropy} We now describe how to compute
the limiter to enforce the local minimum principle on the specific
entropy. \cite{Khobalatte_Perthame_1994} is the first paper we are
aware of where this type of limiting is done.

By virtue of Lemma~\ref{Lem:quasiconcave} and
Lemma~\ref{Lem:compute_lij}, we have to estimate
$\limiter_j^{i,s}\in [0,\limiter_j^{i,e}]$ so that
$\Psi(\bsfU_i^{L,n+1} + \limiter \bsfP_{ij})\ge 0$ for all
$\limiter\in [0, \limiter_j^{i,s}]$, with
$\Psi(\bsfU) := \Phi(\bsfU) - s_i^{\min}$, where we recall that
$\Phi(\bsfU):=s(\rho,e(\bsfU))$ is the specific entropy as a function of the
conservative variables. 
 
\begin{lemma} Let $t_0$ be defined as follows: 
  \textup{(i)} If $\Psi(\bsfU_i^{L,n+1}+ \bsfP_{ij})\ge 0$, then we
  set $t_0=1$;
  \textup{(ii)} If $\Psi(\bsfU_i^{L,n+1}+ \bsfP_{ij})<0$ and
  $\Psi(\bsfU_i^{L,n+1})>0$, we set $t_0$ to be the unique positive
  root to the equation $\Psi(\bsfU_i^{L,n+1} + t \bsfP_{ij})=0$;
  \textup{(iii)} If $\Psi(\bsfU_i^{L,n+1}+ \bsfP_{ij})<0$ and
  $\Psi(\bsfU_i^{L,n+1})=0$, the equation
  $\Psi(\bsfU_i^{L,n+1} + t \bsfP_{ij})=0$ has exactly two roots
  (possibly equal) and we take $t_0$ to be the largest nonnegative
  root.  More precisely, if $D\psi(\bsfU_i^{L,n+1})\cdot \bsfP_{ij}\le0$ then $t_0=0$,
  and if $D\psi(\bsfU_i^{L,n+1})\cdot \bsfP_{ij}>0$ then $t_0>0$ is the unique positive root
  of  $\Psi(\bsfU_i^{L,n+1} + t \bsfP_{ij})=0$ and has to be computed.
  Then setting $\limiter_j^{i,s}= \min(t_0,\limiter_j^{i,e})$,
  we have $\Psi(\bsfU_i^{L,n+1}+ \limiter \bsfP_{ij})\ge 0$ for all
  $\limiter \in [0, \limiter_j^{i,s}]$.
\end{lemma}
\begin{proof} Let us first observe that the equation
  $\Psi(\bsfU_i^{L,n+1}+ t\bsfP_{ij})=0$ has exactly two roots
  (possibly equal) because the upper level set
  $L_0=\{\bsfU\st \Psi(\bsfU)\ge 0$ is convex and any line that
  intersects the upper level set crosses the boundary at two points,
  say $t_-\le t_+$, ($t_-= t_+$ when the line is tangential to the
  boundary of the upper level set). Note that $t_{-} \le 0$ since
  $\Psi(\bsfU_i^{L,n+1})\ge 0$.
  \textup{(i)} If $\Psi(\bsfU_i^{L,n+1}+ \bsfP_{ij})\ge 0$, then
  $t_+\ge 1$ and the entire segment
  $\{\bsfU_i^{L,n+1}+ t\bsfP_{ij}\st t\in [0,t_0=1]\}$ in $L_0$ by
  convexity.
  \textup{(ii)} If $\Psi(\bsfU_i^{L,n+1}+ \bsfP_{ij})<0$ and
  $\Psi(\bsfU_i^{L,n+1})>0$, then $t_+\in (0,1)$, and upon setting $t_0=t_+$,
 the entire segment
  $\{\bsfU_i^{L,n+1}+ t\bsfP_{ij}\st t\in [0,t_0]\}$ in $L_0$ by
  convexity.
  \textup{(iii)}  If $\Psi(\bsfU_i^{L,n+1}+ \bsfP_{ij})<0$ and
  $\Psi(\bsfU_i^{L,n+1})=0$, there are two possibilities: (i) $D\psi(\bsfU_i^{L,n+1})\cdot \bsfP_{ij}\le 0$; 
  and (ii) $D\psi(\bsfU_i^{L,n+1})\cdot \bsfP_{ij}>0$. If $D\psi(\bsfU_i^{L,n+1})\cdot \bsfP_{ij}\le 0$ then 
  by convexity $\Psi(\bsfU_i^{L,n+1}+ t\bsfP_{ij})<0$ for all $t>0$. Hence,    $t_+=0$ is the largest 
  nonnegative toot of the equation $\Psi(\bsfU_i^{L,n+1}+ t\bsfP_{ij})=0$ and therefore $t_0=t_+=0$.
   In the remaining case, $D\psi(\bsfU_i^{L,n+1})\cdot \bsfP_{ij}>0$, we have that $0<t_+<1$ and 
   $t_0=t_+$.
\end{proof}

Let us now explain how the above line search can be done efficiently. Thermodynamic principles 
imply that there exists a smooth function $g: \Real_{+} \CROSS\Real\to\Real_{+}$ such that 
$\rho e = g(s,\rho)$. Note that
the identity $\partial_s e(\rho,s) = \frac{1}{\partial_e s(\rho,e)}$
together with the fundamental thermodynamic inequality $\partial_e s(\rho,e)>0$,
which is equivalent to the temperature being positive, implies
$\partial_s g(\rho,s)>0$.
 Since $\partial_s g > 0$, the minimum
principle on the specific entropy $\Psi(\bsfU) := \Phi(\bsfU) - s_i^{\min}\ge 0$ is equivalent to enforcing
$ g(s,\rho)=\rho e \ge g(s_i^{\min},\rho)$.

When the function $g(s,\rho)$ satisfies $\partial_{\rho\rho} g \ge 0$
and up to a change of notation, the above constraint can be
reformulated as follows:
$\Psi(\bsfU):= \rho e(\bsfU) - g(s_i^{\min},\rho) \ge 0$.  Note that
$\partial_{\rho\rho} g \ge 0$ implies that the new function
$\Psi(\bsfU)$ is concave with respect to the conservative variables;
as a result, the line search
$h(t):=\psi(\bsfU_i^{L,n+1}+ t\bsfP_{ij})=0$ can be done efficiently
because $h$ is concave and $h(0)\ge0$. If $h(1)\ge0$ we set $t_0=1$
and if $h(1) < 0$ we use a combination of secant and Newton method to
find the unique $0\le t_0 <1$ such that $h(t_0)=0$.
For example, the covolume equation of state falls into this category
since in this case we have
$\rho e = \frac{\rho^\gamma}{(1-b\rho)^{\gamma-1}}
\exp((\gamma-1)(s-s_0))=:g(s,\rho)$.
This is also the case for the stiffened gas equation of state,
$\rho e = e_0\rho + p_\infty(1-b\rho) + 
\frac{\rho^\gamma}{(1-b\rho)^{\gamma-1}}
\exp((\gamma-1)(s-s_0))
=:g(s,\rho)$ where $e_0$, $s_0$ and $p_\infty$ are are constant coefficients
characteristic of the thermodynamic properties of the fluid, see \cite{Mettayer_Saurel_2016}
for details.

In the general case, \ie when $g(s,\rho)$ does not satisfy
$\partial_{\rho\rho} g \ge 0$, we can use a different strategy for
imposing $\Psi(\bsfU) = \Phi(\bsfU) - s_i^{\min} \ge 0$.  Namely,
using that $\rho > 0$ and using again a change of notation, we
transform the constraint to
$\Psi(\bsfU):=\rho\Phi(\bsfU) - s_i^{\min} \rho \ge 0$.  Note that the
function $-\rho\Phi(\bsfU)$ is a mathematical entropy for the Euler
system and under the standard assumptions (hyperbolicity and positive
temperature) it is convex, see
\cite[Thm.~2.1]{Harten_Lax_Levermore_Morojoff_1998}.  Therefore, the
line search $h(t):=\psi(\bsfU_i^{L,n+1}+ t\bsfP_{ij})=0$ can be done
efficiently because $h$ is concave and $h(0)\ge0$. If $h(1)\ge0$ we
set $t_0=1$ and if $h(1) < 0$ we use a combination of secant and
Newton method to find the unique $0\le t_0 <1$ such that $h(t_0)=0$.

\subsection{Relaxation} \label{Sec:relaxation} It is observed in
\cite[\S3.3]{Khobalatte_Perthame_1994} that strictly enforcing the
minimum principle on the specific entropy degrades the converge rate
to first-order; it is said therein that ``It seems impossible to
perform second-order reconstruction satisfying the conservativity
requirements $\ldots$ and the maximum principle on $S$''. We have also
observed this phenomenon.  Moreover, it is well known that, when
applied to scalar conservation equations, limitation (in some broad
sense) reduces the accuracy to first-order near maxima and minima of
the solution. One typical way to address this issue in the finite
volume literature consists of relaxing the slope reconstructions; see
\cite{Harten_Osher_1987}, \cite[\S2.1]{Schmidtmann_Abgrall_Torrilhon_2016}. In the
present context, since we do not have any slope to reconstruct, we are
going to relax the constraints so that the violation of the constraint
is second-order accurate.

\subsubsection{Relaxation on the density and the internal energy}
Let us denote by $\varrho$ one of the quantities that we may want to limit
from below, excluding the specific entropy, say $\rho$, $-\rho$, or
$(\rho e)$ and let $\varrho^{\min}$ be the corresponding bound given
by the technique described in \S\ref{Sec:bounds}. For each
$i\in \calI$, we set
$\Delta^2 \varrho_i^n := \sum_{i\ne j\in \calI(\Dom_i)} \varrho_i^n -
\varrho_j^n$ and we define
\begin{align}
\overline{\Delta^2 \varrho_i^n} &:= \frac{1}{2\text{card}(\calI(\Dom_i))}
\sum_{i\ne j\in \calI(\Dom_i)} (\frac12\Delta^2 \varrho_i^n +\frac12\Delta^2\varrho_j^n),\\
\widetilde{\Delta^2 \varrho_i^n} &:= \text{minmod}\Big\{\frac12\Delta^2\varrho_j^n \st j\in\calI(\Dom_i)\Big\},
\end{align}
where the minmod function of a finite set is defined to be zero if
there are two numbers of different sign in this set, and it is equal
to the number whose absolute value is the smallest otherwise. Then we
propose two types of relaxation defined as follows:
\begin{align}
  \overline{\varrho^{\min}_i} 
= \max(0.99\varrho^{\min}_i, \varrho^{\min}_i - |\overline{\Delta^2 \varrho_i^n}|) \label{average_relaxation}\\
  \widetilde{\varrho^{\min}_i} 
= \max(0.99\varrho^{\min}_i, \varrho^{\min}_i - |\widetilde{\Delta^2 \varrho_i^n}|).
\end{align}
When doing limitation we use either $\overline{\varrho^{\min}_i}$ or
$\widetilde{\varrho^{\min}_i}$ instead of $\varrho^{\min}_i$. It is
shown in the numerical section that both relaxations are robust.  

\begin{remark}[relaxation vs. no relaxation]
  We have observed numerically that the proposed method is
  second-order accurate in the $L^1$-norm without relaxation if
  limitation is done on the density and the internal energy.
Relaxation is necessary only to get second-order accuracy in the $L^\infty$-norm.
We have observed though that the minmod relaxation
is slightly more restrictive than the other one since it does
not deliver second-order accuracy in the maximum norm; only the
averaging relaxation \eqref{average_relaxation} has been found to give
second-order in the $L^\infty$-norm.
\end{remark}

\begin{remark}[positivity]
  Note that the somewhat ad hoc threshold $0.99\varrho^{\min}_i$ in the
  above definitions imply that one never relaxes more than 1\% of the
  legitimate lower bound. In particular, when applied to the density or
  the internal energy the above relaxation
  guarantees positivity of the density and the internal energy.
\end{remark}

\subsubsection{Relaxation on the specific entropy} \label{Sec:relax_specific_entropy}
We proceed as in \cite[\S3.3]{Khobalatte_Perthame_1994} and relax the
lower bound on the specific entropy more aggressively since in smooth
regions this function is constant. Let $\varrho$ be the quantity
associated with the constraint on the specific entropy; it could be
$s$, $\exp(s)$, or $\rho e/\rho^\gamma$ in the case of polytropic
gases, depending on the way one chooses to enforce the minimum
principle on the specific entropy (see \S\ref{Sec:specific_entropy}).
Let $\bx_{ij}=\frac12(\bx_i+\bx_j)$. We measure the local variations
of $\varrho$ by setting
$\Delta \varrho_i^n = \max_{i\ne j\in\calI(\Dom_i)}
(\varrho^n(\bx_{ij}) - \varrho_i^{\min})$
and we relax $\varrho_i^{\min}$ by setting
\begin{equation}
\overline{\varrho_i^{\min}} 
= \max(0.99\varrho_i^{\min}, \varrho_i^{\min} - \Delta \varrho_i^n).\label{relax_specific_entropy}
\end{equation}
Note that contrary to the appearances, and as already observed in
\citep{Khobalatte_Perthame_1994}, the size of the relaxation is
$\calO(h^2)$: In the vicinity of shocks, there is no need for
relaxation since the first-order viscosity takes over and thereby
makes the solution minimum principle preserving on the specific
entropy. In smooth regions, \ie isentropic regions, the specific
entropy is constant and $\Delta \varrho_i^n$ measures the local
curvature of $s$ induced by the nonlinearity of $s$, that is to say
$\Delta \varrho_i^n$ is $\calO(h^2)$.
 
\begin{remark}[positivity]
  The threshold $0.99\varrho^{\min}_i$ guarantees positivity of the
  internal energy. We have observed numerically that this threshold is
  never reached when the mesh is reasonably fine enough.
\end{remark}

\section{Numerical illustrations}
\label{Sec:Numerical_illustrations}

We report in this section numerical tests we have done to
illustrate the performance of the proposed method. All the tests are
done with the equation of state $p=(\gamma-1)\rho e$, \ie
$(\gamma-1) s(\rho,e) = \exp(\rho e/\rho^\gamma)$.
\subsection{Technical details}
Three different codes implementing the method described in the paper
have been written to ensure reproducibility. Limiting is done only
once in the three codes, \ie $k_{\max}=1$.

The first one, henceforth called Code~1, does not use any particular
software.  It is based on Lagrange elements on simplices. This code
has been written to be dimension-independent, \ie the same data
structure and subroutines are used in one dimension and in two
dimensions. The two-dimensional meshes used for Code~1 are nonuniform
triangular Delaunay meshes. All the computations reported in the paper
are done with continuous $\polP_1$ elements. The high-order method
uses the entropy viscosity commutator described in
\eqref{def_of_Rin}-\eqref{def_dij_EV} with the entropy
$p^{\frac1\gamma}$.  Limitation on the density is enforced by using
the technique described in \S\ref{Sec:density} . The bounds on the
density are relaxed by using the averaging technique described in
\S\ref{Sec:relaxation}.  The minimum principle on the specific entropy
$\exp((\gamma-1)s) \ge \exp((\gamma-1)s^{\min})$ is enforced by using
the method described in \S\ref{Sec:specific_entropy} with the
constraint $\Psi(\bsfU):= \rho e - \varrho^{\min} \rho^\gamma\ge 0$,
where we recall that $\rho e/\rho^\gamma = \exp((\gamma-1)s)$.  The
lower bound on the specific entropy is defined by using
$\varrho^{\min}_i = \min_{j\in \calI(\Dom_i)} \rho_i^n
e_i^n/(\rho_i^n)^\gamma$
(instead of \eqref{smindef}) and $\varrho^{\min}_i$ is relaxed by
using \eqref{relax_specific_entropy} with
$\varrho= \rho e/\rho^\gamma$. No limitation on the internal energy is
applied in Code~1; the positivity of the internal energy is guaranteed
by the minimum principle on the specific entropy.

The second code, henceforth called Code~2, is based on the open-source
finite element library deal.II, see \cite{dealII85,
  BangerthHartmannKanschat2007}. The tests reported in the paper are
done with continuous $\polQ_1$ (quadrilateral) elements. The code is
written in a dimension-independent fashion and all the computational
tasks (\eg assembly, linear solvers and output) are implemented in
parallel via MPI. This code implements the smoothness-based high-order
graph viscosity described in
\S\ref{Sec:smoothness_based_viscosity}. All the computations reported
in the paper are done with $g(\bu) = S(\bu)$ and
$\psi(\alpha)=\alpha^4$. (We have verified that the other choices for
$g(\bu)$ mentioned in \S\ref{Sec:smoothness_based_viscosity} produce
comparable results.) As stated at the beginning of
  \S\ref{Sec:smoothness_based_viscosity}, this method introduces
  additional diffusion close to local extrema (whether smooth or
  not). While this does not affect the second-order decay rates in the
  $L^1$-norm, it degrades the accuracy in the $L^\infty$-norm.
  Limitation is done on the density and the
  specific entropy exactly as explained above for Code~1.  The minimum
  principle on the specific entropy is relaxed as explained in
  \S\ref{Sec:relax_specific_entropy}, but no relaxation is applied on
  the density bounds.

The last code, henceforth called Code~3, uses the open-source finite
element library FEniCS, see \eg \cite{Fenics2003}, and the
computations are done on simplices. The implementation in FEniCS is
independent of the space dimension and the polynomial degree of the
approximation. The library is fully parallel. All the numerical
integrations are done exactly by automatically determining the
quadrature degree with respect to the complexity of the underlying
integrand and the polynomial space. The results reported in the paper
use the entropy viscosity method with the entropy
$p^{\frac1\gamma}$. The limitation and bound relaxation is done on the density and the
specific entropy exactly like in Code~1. We refer the reader to the
description of Code~1 for the details.

The time stepping is done in the three codes by using the SSP(3,3)
method (three stages, third-order), see
\cite[Eq.~(2.18)]{Shu_Osher1988} and
\cite[Thm.~9.4]{Kraaijevanger_1991}. The time step is recomputed at
each time step by using the formula
$\dt = \text{CFL}\CROSS \min_{i\in \calI} \frac{m_i}{|d_{ii}^{L,n}|}$
with $d_{ii}^{L,n}=-\sum_{ij} d_{ij}^{L,n}$ given in
\eqref{Def_of_dij_I}.

When working with manufactured solutions, for $p\in [1,\infty]$, we
compute a consolidated error indicator at time $t$ by adding the
relative error in the $L^p$-norm of the density, the momentum, and the
total energy as follows:
\begin{align}
  \delta_p(t):=\frac{\|\rho_h(t)-\rho(t)\|_{L^p(\Dom)}}{\|\rho(t)\|_{L^p(\Dom)}} +
  \frac{\|\bbm_h(t)-\bbm(t)\|_{\bL^p(\Dom)}}{\|\bbm(t)\|_{\bL^p(\Dom)}} +
  \frac{\|E_h(t)-E(t)\|_{L^p(\Dom)}}{\|E(t)\|_{L^p(\Dom)}}. \label{def_delta_t}
\end{align}
As some tests may exhibit superconvergence effects we also consider
the fully discrete consolidated error indicator $\delta_p^\Pi(t)$
defined as above with $\rho(t)$, $\bbm(t)$, and $E(t)$ replaced by
$\Pi_h\rho(t)$, $\Pi_h\bbm(t)$, and $\Pi_hE(t)$, where $\Pi_h$ is the
Lagrange interpolation operator.

\subsection{1D smooth wave}
We start with a one-dimensional test whose purpose is to estimate
the convergence rate of the method with a very smooth solution. We
consider the following exact solution to the Euler equations: $v(x,t) = 1$, $p(x,t) = 1$ and
\begin{align}
\rho(x,t) = \begin{cases}  
1 + 2^6(x_1-x_0)^{-6}(x-t-x_0)^3 (x_1-x+t)^3 & \text{if $x_0\le x-t< x_1$}\\
1 & \text{otherwise}
\end{cases}
\end{align}
with $x_0=0.1$, $x_1=0.3$ and $\gamma=\frac{7}{5}$. The computational
domain is $\Dom=(0,1)$ and the computation is done from $t=0$ to
$t=0.6$. The consolidated error indicator in the maximum norm
$\delta_\infty(t)$ is reported in Table~\ref{Table:Leblanc_num}. Note
that we report the discrete error indicator $\delta_\infty^\Pi(t)$ for
Code~1 (based on the Entropy Viscosity Method) in order to show that
we obtain $\calO(h^3)$ super-convergence in compliance with the
theoretical result stated in \cite[Prop~2.2]{Guermond_pasquetti_2013}.
Code~2, which we recall is based in the smoothness of the mathematical
entropy, delivers $\calO(h^{1.5})$ as expected due to clipping effects
induced by the smoothness indicator.

\begin{table}[ht]\small \centering
  \caption{1D smooth wave, $\mathbb{P}_1$ meshes, Convergence tests with Code~1 and Code~2, $\text{CFL}=0.25$.}
  \label{Table:1D_wave_num} 
  \begin{tabular}{|c|c|c|c|c|}
    \hline
    \multirow{2}{*}{\# dofs} 
 & \multicolumn{2}{|c|}{Code~1} &  \multicolumn{2}{|c|}{Code~2} \\ \cline{2-5} 
 &  $\delta_\infty^\Pi(t)$ & rate    &  $\delta_\infty(t)$ & rate \rule{0pt}{2.5ex}\\ \hline
100	& 6.34E-02 &     & 2.32E-01 &      \\ \hline
200	& 1.62E-02 &1.96 & 8.30E-02 & 1.48 \\ \hline
400	& 2.69E-03 &2.59 & 2.87E-02 & 1.53 \\ \hline
800	& 3.74E-04 &2.85 & 9.66E-03 & 1.57 \\ \hline
1600	& 4.62E-05 &3.02 & 3.22E-03 & 1.58 \\ \hline
3200	& 5.89E-06 &2.97 & 1.07E-03 & 1.58 \\ \hline
6400	& 7.37E-07 &3.00 & 3.74E-04 & 1.52 \\ \hline
\end{tabular}
\end{table}

\subsection{Rarefaction wave} \label{Sec:rarefaction} We now consider
a Riemann problem with a solution whose components are all continuous
and whose derivatives have bounded variations. The best-approximation
error in the $L^1$-norm on quasi-uniform meshes is then $\calO(h^2)$.
The Riemann problem in question has the following data:
$(\rho_L,v_L,p_L)=(3,c_L,1)$,
$(\rho_R,v_R,p_R)=(\frac12,v_L+\frac{2}{\gamma-1}(c_L-c_R),
p_L(\frac{\rho_R}{\rho_L})^\gamma)$,
where $c_L = \sqrt{\gamma p_L/\rho_L}$,
$c_R = \sqrt{\gamma p_R/\rho_R}$. The equation of state is a gamma-law
with $\gamma = \frac{7}{5}$. The exact solution to this problem is a
rarefaction wave which can be constructed analytically, see \eg
\cite[\S4.4]{Toro_2009}. The solution is given in
Table~\ref{Table:rarefaction}. In this table, the ratio
$\xi:=\frac{x-x_0}{t}$ is the self-similar variable, where $x_0$ is
the location of the discontinuity at $t=0$.  This problem is quite
challenging for any method enforcing the minimum principle on the
specific entropy. We have observed that the convergence rate on this problem reduces
to $\calO(h)$ if the minimum on the specific entropy is not relaxed
(see also \cite[\S3.3]{Khobalatte_Perthame_1994}).
\begin{table}[ht]\caption{Solution to the rarefaction
    wave.}\label{Table:rarefaction}\centering
\begin{tabular}{|c|c|c|c|} \hline
 & $\xi\le v_L-c_L$& $v_L-c_L< \xi\le v_R-c_R$& $v_R-c_R<\xi$\\ \hline
$\rho$&$\rho_L$&$\rho_L \big(\frac{2}{\gamma+1}
  + \frac{\gamma-1}{\gamma+1} \frac{v_L-\xi}{c_L}\big)^{\frac{2}{\gamma-1}}$ & $\rho_R$\\\hline
$v$ & $v_L$ & $\frac{2}{\gamma+1}(c_L + \frac{\gamma-1}{2}v_L +\xi)$& $v_R$ \\\hline
$p$ & $p_L$&  $p_L\big(\frac{2}{\gamma+1}
  + \frac{\gamma-1}{\gamma+1}\frac{v_L-\xi}{c_L}\big)^{\frac{2\gamma}{\gamma-1}}$ & $p_R$\\ \hline
\end{tabular}
\end{table}

We run Code~1, Code~2, and Code~3 on the computational domain
$\Dom=(0,1)$ with $x_0=0.2$ and the initial time is
$t_0=\frac{0.2}{v_R-c_R}$.  The initial data is the exact solution at
$t=\frac{0.2}{v_R-c_R}$ given in Table~\ref{Table:rarefaction}.  The
simulation are run until $t=0.5$. The consolidated error indicator
$\delta_1(t)$ defined in \eqref{def_delta_t} is reported in
Table~\ref{Table:rarefaction_num} for Code~1 and Code~2 only for brevity.
This series of tests shows that the proposed method, with limitation
of the density and the specific entropy as described in
\S\ref{Sec:density}-\S\ref{Sec:specific_entropy}, converges with rate
at least $\calO(h^{1.5})$ on the rarefaction wave problem. We
observe that the low-order method is indeed asymptotically first-order
(the rate is $0.96$ for $12800$ grid points). 

\begin{table}[htb]\small \centering
  \caption{Rarefaction wave, $\mathbb{P}_1$ meshes, 
    Convergence tests with Code~1 and Code~2, $\text{CFL}=0.25$.}
  \label{Table:rarefaction_num} 
  \begin{tabular}{|c|c|c|c|c|c|c|c|c|}
    \hline
    \multirow{2}{*}{\# dofs} 
 & \multicolumn{2}{|c|}{Code~1} & \multicolumn{2}{|c|}{Code~2} & \multicolumn{2}{|c|}{Galerkin}  & \multicolumn{2}{|c|}{Low-order} \\ \cline{2-9} 
        & $\delta_1(t)$ & rate & $\delta_1(t)$ & rate   & $\delta_1(t)$ & rate & $\delta_1(t)$ & rate \\ \hline
100	& 1.01E-03    &      & 3.33E-03    &        & 1.44E-03    &      & 5.10E-02    &      \\
200	& 3.28E-04    & 1.62 & 1.08E-03    & 1.61 & 4.38E-04    & 1.71 & 2.96E-02    & 0.78 \\
400	& 1.13E-04    & 1.54 & 3.57E-04    & 1.61 & 1.42E-04    & 1.62 & 1.68E-02    & 0.82 \\
800	& 4.02E-05    & 1.49 & 1.18E-04    & 1.60 & 4.73E-05    & 1.59 & 9.23E-03    & 0.86 \\
1600	& 1.44E-05    & 1.48 & 3.96E-05    & 1.58 & 1.60E-05    & 1.57 & 4.96E-03    & 0.89 \\
3200	& 5.11E-06    & 1.49 & 1.31E-05    & 1.59 & 5.47E-06    & 1.55 & 2.62E-03    & 0.92 \\
6400	& 1.75E-06    & 1.54 & 4.32E-06    & 1.61 & 1.82E-06    & 1.59 & 1.37E-03    & 0.94 \\
12800	& 5.71E-07    & 1.62 & 1.38E-06    & 1.64 & 5.83E-07    & 1.64 & 7.05E-04    & 0.96 \\ \hline
    \end{tabular}
\end{table}

When comparing the results from Code~1 with the Galerkin solution, we
observe that the extra dissipation induced in Code~1 by the entropy
viscosity and limitation is of the order of the truncation error,
which is optimal. That is, the method does not introduce any
extraneous dissipation on smooth solutions. Note in passing that
  the maximum norm error indicator $\delta_\infty(t)$ has also been computed
  for this test (not reported here for brevity), yielding
  the rate $\calO(h)$ for Code~1 and  $\calO(h^{0.75})$ for
  Code~2. The rate $\calO(h)$ is optimal since the solution is in
  $\bW^{1,\infty}(\Dom)$

\subsection{Leblanc shocktube}
We continue with a challenging Riemann problem that is known in the
literature as the Leblanc shocktube.  The data are as follows:
$(\rho_L, v_L, p_L) = (1, 0, (\gamma-1)10^{-1})$ and
$(\rho_R, v_R, p_R) = (10^{-3}, 0, (\gamma-1) 10^{-10})$ and the
equation of state is a gamma-law with $\gamma = \frac{5}{3}$.  The
exact solution is described in the Table~\ref{Table:Leblanc}.
Denoting by $x_0$ the location of the discontinuity at $t=0$, the
quantity $\xi=\frac{x-x_0}{t}$ is the self-similar variable and the
other numerical values in the table are given with 15 digit accuracy
by $\rho^{*}_{L} = 5.40793353493162\CROSS 10^{-2}$,
$\rho^{*}_{R} = 3.99999806043000\CROSS 10^{-3}$,
$v^*=0.621838671391735$, $p*=0.515577927650970\CROSS 10^{-3}$,
$\lambda_{1}=0.495784895188979$, $\lambda_3 =0.829118362533470$.
\begin{table}[h] 
\caption{Solution of the Leblanc shocktube.} \label{Table:Leblanc}\centering
\begin{tabular}{|c|c|c|c|c|c|c|}\hline
&$\xi\leq-\frac{1}{3}$&$-\frac{1}{3}<\xi\leq\lambda_1$&$\lambda_1<\xi\leq v^*$&$v^*<\xi\leq\lambda_3$& $\lambda_3< \xi$ \\ \hline
$\rho$&$\rho_L$&$(0.75 - 0.75\xi)^{3}$        &$\rho^{*}_{L}$&$\rho^{*}_{R}$&$\rho_R$\\\hline
$v$   &$v_L$   &$0.75(\frac{1}{3}+\xi)$       &$v^*$        &$v^*$        &$v_R$\\\hline
$p$   &$p_L$   &$\frac{1}{15}(0.75-0.75\xi)^5$&$p^*$        &$p^*$        &$p_R$ \\\hline
\end{tabular}
\end{table}

We do simulations with $x_0=0.33$ until $t=2/3$ with Code~2 and
Code~3. The consolidated error indicator $\delta_1(t)$ is reported in
Table~\ref{Table:Leblanc_num}. The convergence rate on $\delta_1(t)$
is clearly first-order which is optimal for this problem.

\begin{table}[ht]\small \centering
  \caption{Leblanc shocktube, $\mathbb{P}_1$ meshes, Convergence tests with Code~1 and Code~2, $\text{CFL}=0.25$.}
  \label{Table:Leblanc_num} 
  \begin{tabular}{|c|c|c|c|c|c|c|}
    \hline
    \multirow{2}{*}{\# dofs} 
 & \multicolumn{2}{|c|}{Code~1} & \multicolumn{2}{|c|}{Code~2} & \multicolumn{2}{|c|}{Low-order} \\ \cline{2-7} 
      & $\delta_1(t)$ & rate    & $\delta_1(t)$  & rate & $\delta_1(t)$ & rate\\ \hline
100   & 1.26E-01      &         & 1.49E-01       &      & 2.61E-01      &     \\
200   & 7.67E-02      & 0.71    & 9.01E-02       & 0.72 & 1.94E-01      & 0.43 \\
400   & 4.31E-02      & 0.83    & 4.92E-02       & 0.87 & 1.41E-01      & 0.46 \\
800   & 2.25E-02      & 0.94    & 2.61E-02       & 0.91 & 9.95E-02      & 0.50 \\
1600  & 1.13E-02      & 0.99    & 1.34E-02       & 0.96 & 6.74E-02      & 0.56 \\
3200  & 5.73E-03      & 0.99    & 6.83E-03       & 0.97 & 4.40E-02      & 0.62 \\
6400  & 2.85E-03      & 1.01    & 3.42E-03       & 0.99 & 2.78E-02      & 0.66 \\
12800 & 1.43E-03      & 1.00    & 1.70E-03       & 1.00 & 1.73E-02      & 0.68 \\ \hline
\end{tabular}
\end{table}

\subsection{Sod, Lax, Blast wave}
We now illustrate the method on a series of traditional problems
without giving the full tables for the convergence rates for
brevity. We consider the Sod shocktube, the Lax shocktube, and the
Woodward-Collela blast wave. We refer the reader to the literature for
the initial data for these test cases.  The computations are done on
the domain $\Dom=(0,1)$ with $\text{CFL}=0.5$ on four different girds
with Code~1. The final times are $t=0.225$ for the Sod shocktube,
$t=0.15$ for the Lax shocktube, and $t=0.038$ for the Woodward-Collela
blast wave. We show the graph of the density for these three cases and
for the four meshes in Figure~\ref{Fig:Sod_Lax_Blast}.  We have
observed the convergence rate to be between $\calO(h^{0.9})$ and
$\calO(h)$ on $\delta_1(t)$ for both the Sod and the Lax shocktubes,
which is near optimal (results not reported for brevity).

\begin{figure}[htb]
\includegraphics[width=0.32\textwidth,trim={20 7 18 10},clip]{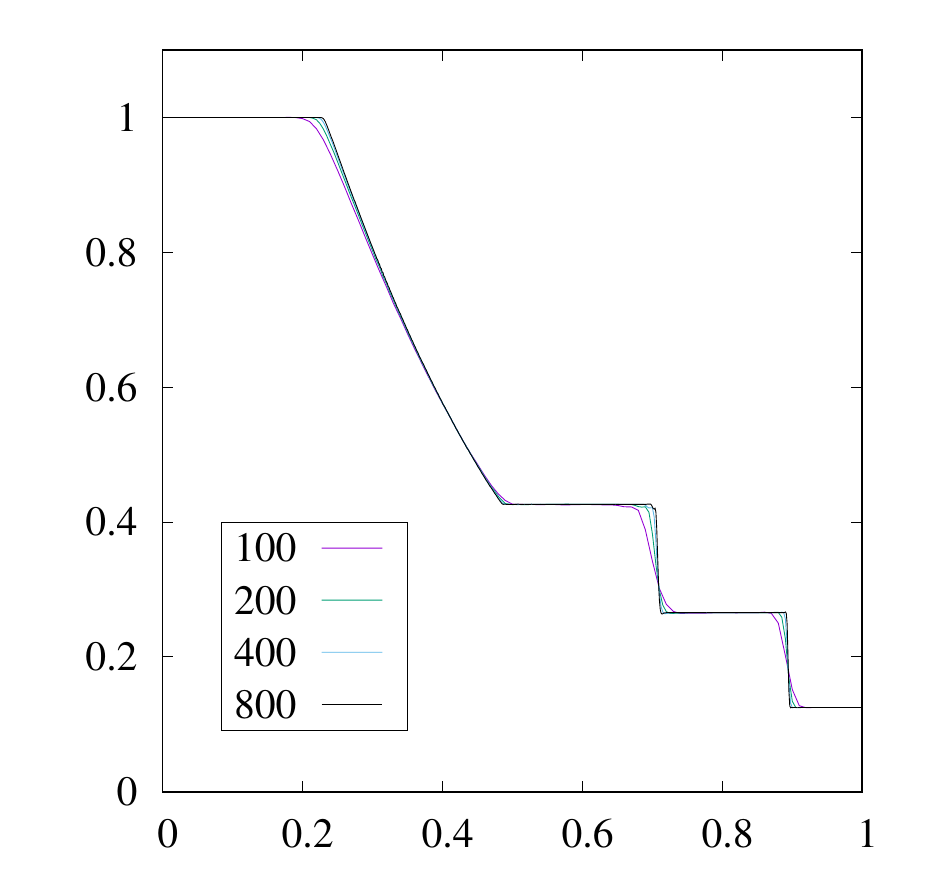}\hfil
\includegraphics[width=0.32\textwidth,trim={20 7 18 10},clip]{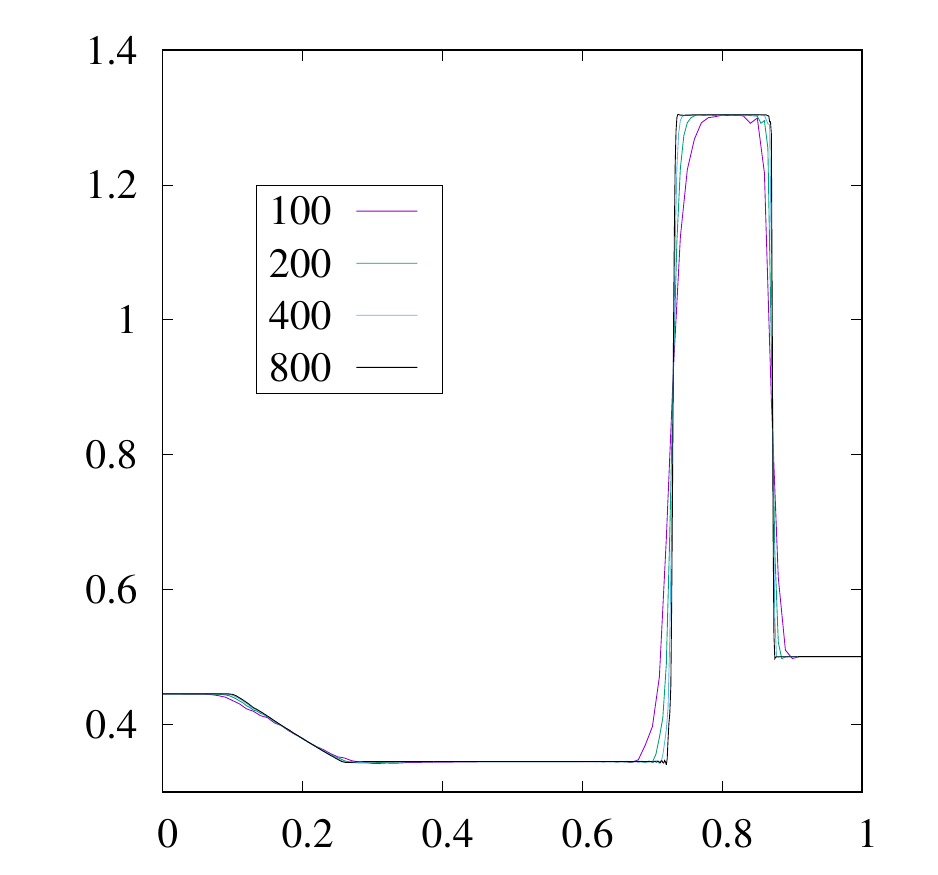}\hfil
\includegraphics[width=0.32\textwidth,trim={20 7 18 10},clip]{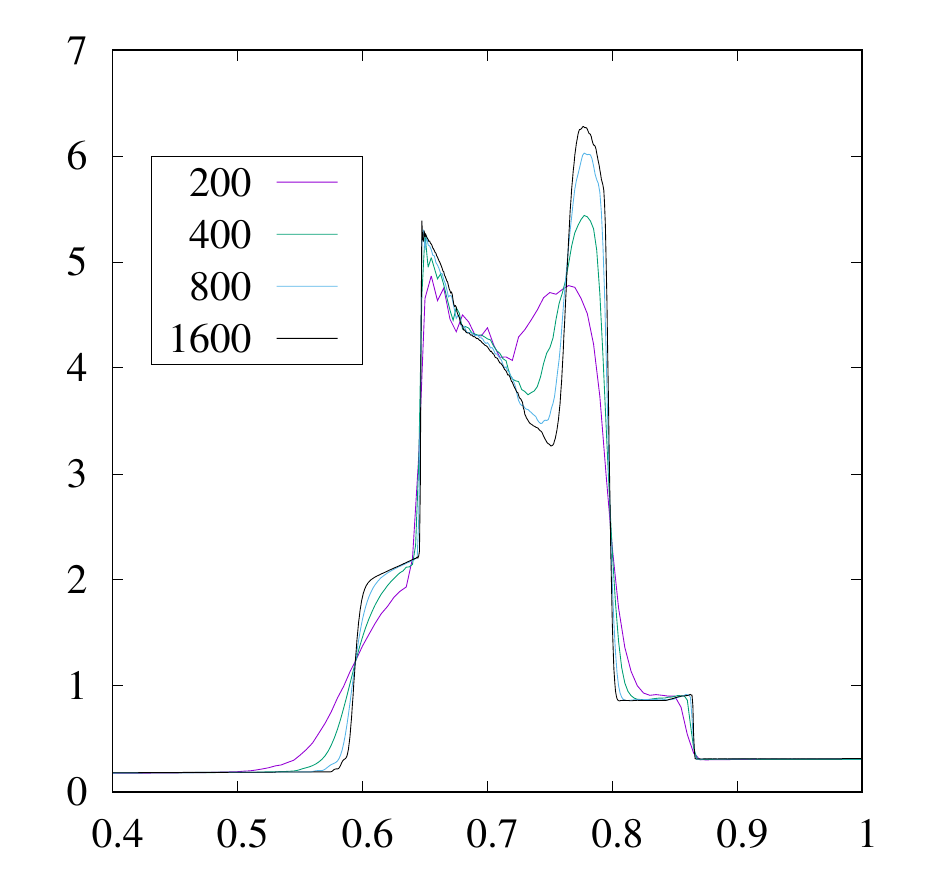}
\caption{Code~1, $CFL=0.5$. Left: Sod shocktube, $t=0.225$; Center: Lax shocktube,
  $t=0.15$; Right: Woodward-Collela blast wave, $t=0.038$.}
\label{Fig:Sod_Lax_Blast}
\end{figure}

\subsection{2D isentropic vortex} \label{Sec:vortex} We now consider a
two-dimensional problem introduced in \cite{Yee_et_al_1999}.  This
test case is often used to assess the accuracy of numerical
schemes. The flow field is isentropic; \ie the solution is smooth and
does not involve any steep gradients or discontinuities.

Let $\rho_\infty = P_\infty = 1$, $\bu_\infty = (2,0)\tr$ be free
stream values. Let us define the following perturbation values for the
velocity and the temperature:
\begin{equation}
    \delta \bu(\bx,t) = \frac{\beta}{2\pi} e^{1-r^2}( - \bar{x}_2, \bar{x}_1),\qquad
    \delta T(\bx,t)   = - \frac{(\gamma-1) \beta^2}{8 \gamma \pi^2} e^{1-r^2},
\end{equation}
where $\beta = 5$ is a constant defining the vortex strength,
$\gamma = \frac{7}{5}$,
$\bar{\bx} = (x_1 - x_{1}^{\text{c}}, x_2 - x_{2}^{\text{c}})$, where
$\bx^{\text{c}} = (x_{1}^{\text{c}}+2t, x_{2}^{\text{c}})$ is the
position of the vortex, and $r^2 = \|\bar\bx\|^2_{\ell^2}$.  The exact
solution is a passive convection of the vortex with
the mean velocity $\bu_\infty $:
\begin{equation}
    \rho(\bx,t) = (T_\infty + \delta T)^{1/{\gamma-1}}, \qquad
    \bu(\bx,t) = \bu_\infty + \delta \bu, \qquad
    p(\bx,t) = \rho^{\gamma}.
\end{equation}

We perform the numerical computation in the rectangle
$\Dom= (-5, 10) \CROSS (-5, 5)$ from $t=0$ until $t=2$.  The initial
mesh consists of $20\CROSS 13$ squares divided by four equilateral
triangles, then the mesh is refined uniformly to compute finer
solutions. The computations are done with Code~3.  The consolidated
error indicator $\delta_\infty(t)$ is reported in
Table~\ref{table:SMOOTH}. Here again we observe second-order accuracy
in the maximum norm.
\begin{table}[ht]\small
  \caption{Isentropic vortex, $\mathbb{P}_1$ meshes,
    convergence tests, $t=2$. Code~3, $\text{CFL}=0.5$.}
  \label{table:SMOOTH} 
\centering
\begin{tabular}{|c|c|c|} \hline
 \# dofs    & $\delta_\infty(t)$ & rate \\ \hline
 2216  & 2.23e-01 &   --   \\ 
 8588  & 4.50e-02 &   2.37 \\ 
34456  & 1.05e-02 &   2.09 \\ 
136748 & 2.82e-03 &   1.91 \\ 
547416 & 9.09e-04 &   1.63 \\ 
     \hline
  \end{tabular} 
\end{table}

\subsection{Mach 3 step} \label{Sec:Mach3} We finish by illustrating
the method on the classical Mach 3 flow in a wind tunnel with a
forward facing step.  The computational domain is
$\Dom=(0,1)\CROSS (0,3)\setminus (0.6,3)\CROSS(0,0.2)$; the geometry
of the domain is shown in Figure~\ref{Fig:Mach3_step}.  The initial
data is $\rho=1.4$, $p=1$, $\bv=(3,0)\tr$. The inflow boundary
conditions are $\rho_{|\{x=0\}}=1.4$, $p_{|\{x=0\}}=1$,
$\bv_{|\{x=0\}}=(3,0)\tr$. The outflow boundary conditions are free,
\ie we do nothing at $\{x=3\}$. On the top and bottom boundaries of
the channel we enforce $\bv\SCAL \bn=0$. This is done by setting
$\bbm_i^{n+1}\SCAL \bn=0$ at the end of each substep of the SSP
RK(3,3) algorithm at each node $\bx_i$ belonging to the boundary in
question; moreover the vectors $\bc_{ij}$ introduced in
\eqref{def_of_cij} are redefined as follows
$\bc_{ij} = -\int_{\Dom} \varphi_j \GRAD \varphi_i \dif x$ at each
node $\bx_i$ belonging to the boundary in question. This integration
by parts is justified by the observation that it implies global
conservation under the assumption that $\bef(\bu)\SCAL \bn =0$ over
the entire boundary of $\Dom$.

The computation is done from $t=0$ to $t=4$.  We show in
Figure~\ref{Fig:Mach3_step} a Schlieren-type snapshot of the density
at $t=4$ obtained with Code~2.  The Kelvin-Helmholtz instability of
the contact discontinuity is clearly visible. No regularization or
smoothing is applied at the top left corner of the step.

\begin{figure}[hbt]
\includegraphics[width=0.93\textwidth]{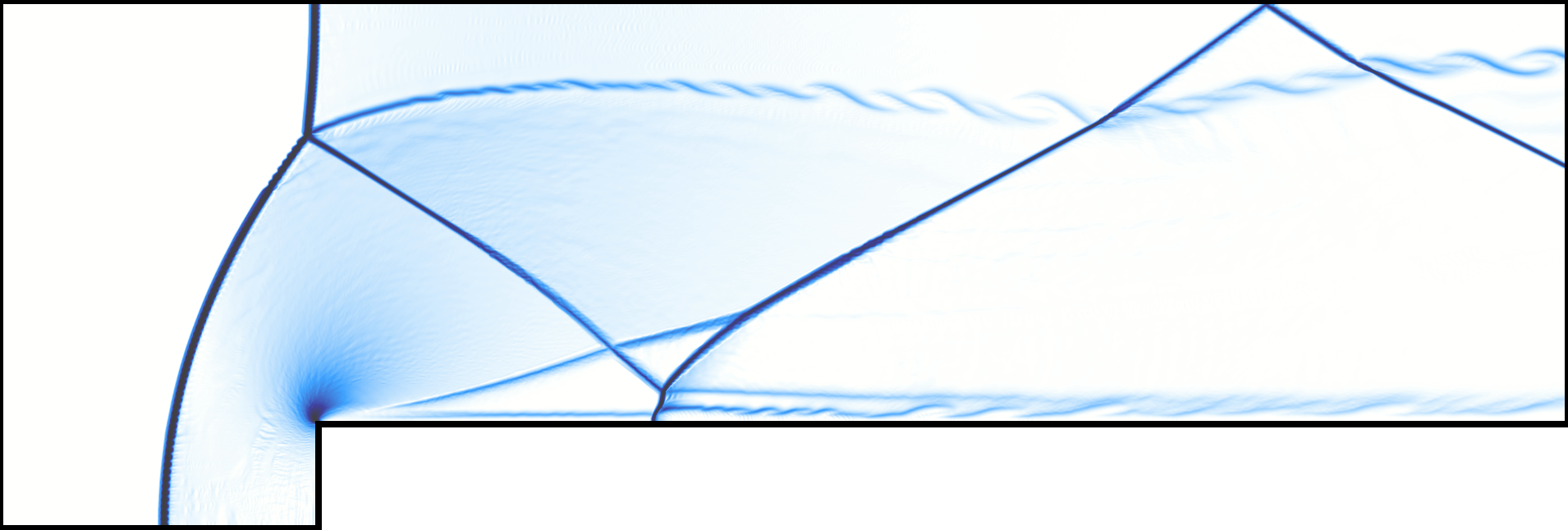}
\caption{Mach 3 step, density, $t=4$, density,
$\polP_1$ approximation using Code~2.}
\label{Fig:Mach3_step}
\end{figure}

\bibliographystyle{abbrvnat}
\bibliography{ref}

\end{document}